\documentclass[12pt]{article}

\usepackage[margin=3cm]{geometry}

\usepackage{amsmath,amsfonts,amssymb, amsthm}

\newcommand\N{\mathbb{N}}
\newcommand\R{\mathbb{R}}
\newcommand\G{\mathcal{G}}
\newcommand\F{\mathcal{F}}
\newcommand\src{\mathsf{s}}
\newcommand\tgt{\mathsf{t}}
\newcommand\mul{\mathsf{m}}
\newcommand\inv{\mathsf{i}}
\newcommand\one{\mathsf{u}}
\newcommand\mset[1]{{#1}^{(2)}}
\newcommand\Lie{\mathrm{Lie}}
\newcommand\anch{\mathsf{a}}

\newcommand\cA{\mathcal{A}}

\newcommand\pr{\operatorname{pr}}
\newcommand\id{\operatorname{id}}

\newcommand\Fl{\operatorname{Fl}}
\newcommand\shcinf{\mathscr{C}^\infty}

\newcommand\SInvMan{\mathrm{SInvMan}}
\newcommand\SInvVBun{\mathrm{SInvVBun}}
\newcommand\SInvLieGpd{\mathrm{SInvLieGpd}}
\newcommand\SInvLieAlg{\mathrm{SInvLieAlg}}
\newcommand\Man{\mathrm{Man}}

\newcommand\wt[1]{\widetilde{#1}}

\newcommand\sis{s.i.s.\ }

\newcommand\strong{strong}
\newcommand\epsij{\epsilon_{ij}}
\newcommand\vpjl{\varprojlim}

\usepackage{enumitem}
\setlist[enumerate]{topsep=2pt,
itemsep=1pt,parsep=0pt,
partopsep=0pt}

\theoremstyle{plain}
\newtheorem{theorem}{Theorem}[section]     
\newtheorem{lemma}[theorem]{Lemma}         
\newtheorem{proposition}[theorem]{Proposition}
\newtheorem{corollary}[theorem]{Corollary}

\theoremstyle{definition}
\newtheorem{definition}[theorem]{Definition}
\newtheorem{example}[theorem]{Example}

\theoremstyle{remark}
\newtheorem{remark}[theorem]{Remark}

\usepackage{titlesec}
\titleformat{\section}
  {\normalfont\large\bfseries}{\thesection}{1em}{}
\titleformat{\subsection}
  {\normalfont\normalsize\bfseries}{\thesubsection}{1em}{}
\titleformat{\subsubsection}
  {\normalfont\normalsize\itshape}{\thesubsubsection}{1em}{}

\usepackage{xcolor}
\usepackage{hyperref}
\hypersetup{
    colorlinks=true,
    linkcolor=blue,      
    citecolor=blue,      
    urlcolor=blue,       
    filecolor=blue       
}

\usepackage{mathrsfs}
\newcommand\shA{\mathscr{A}}

\usepackage{xurl}

\newcommand\Emph[1]{\emph{\textbf
{#1}}}

\title{Inverse limits of locally convex Lie groupoids}
\author{A. G. Shaltut\thanks{ahmed.shaltut@queensu.ca} }
\date{}

\begin{document}

\maketitle

\begin{abstract}
  We study strict inverse systems of locally convex Lie groupoids and Lie algebroids. We show that strict inverse limits of Lie algebroids and Lie groupoids inherit natural Lie structures and establish a natural isomorphism 
  of functors $\vpjl \circ \Lie \cong \Lie \circ \vpjl$.
  We then consider inverse systems of integrable Banach-Lie algebroids and their source-simply connected levelwise integrations. We give conditions under which these integrations form a strict inverse system of Lie groupoids and hence their inverse limit integrates the inverse-limit algebroid.
  Under additional hypotheses on the source-fiber inverse sequence, we describe the obstruction to source-simply connectedness of the resulting integration in terms of the derived inverse limit
$\vpjl\nolimits^{1}\pi_2$
of the second homotopy groups of the source fibers.
  We illustrate the theory with several examples: 
  inverse limits of jet
groupoids, current groupoids, and gauge groupoids, including the
corresponding inverse-limit Atiyah algebroids and a source-simply
connected inverse-limit integration arising from the quaternionic Hopf
bundle.
\end{abstract}

\section{Introduction}

Inverse  limits of geometric structures have been studied extensively in the literature.  For systematic treatments, we refer to
\cite{DGV16,CP23} and the references within.  In particular, inverse limit methods have been
developed for Fr\'echet manifolds, vector bundles, principal bundles, Lie groups, and Lie algebroids. In particular,
Cabau studied strong
inverse limits of Banach Lie algebroids \cite{Cab12},
while Pelletier
studied integrability of inverse limits of Banach distributions 
including applications to inverse sequences of Banach Lie algebroids
\cite{Pel22}.
In this paper, we work with
inverse systems indexed by an arbitrary directed set and do not require the
manifold topology on the inverse limit to coincide with the ordinary
inverse limit topology. We call these \emph{strict inverse
systems}.
We
call a strict system \emph{strong} when the two topologies agree.  Thus the
strong case is contained in the strict one, and in the strong case several
continuity requirements for compatible limit maps become automatic.  
The distinction is not merely formal. Indeed,
the dyadic solenoid is not a locally convex
manifold in its inverse limit topology,
but it admits a strict smooth atlas with a
strictly finer topology; see Section~\ref{sec:mfds_bundles} below.

We adopt the Lie algebroid notion and the dg-map Lie algebroid morphism notion in~\cite{Sha26}. 
A strict inverse system of Lie algebroids has a
canonical inverse-limit Lie algebroid, characterized by the property that
the canonical limit maps are Lie algebroid morphisms.
We employ a notion of strictness for inverse systems of Lie groupoids
that guarantees a canonical Lie groupoid structure on the inverse limit.
The Lie construction is compatible with these limits. Indeed, we show that there is a natural
isomorphism
$\Lie(\vpjl \G_i) \cong \vpjl\Lie(\G_i)$.
We next consider inverse systems of integrable Banach-Lie algebroids. Using source-simply connected integrations at each level and Lie II, we first obtain an inverse system of Banach-Lie groupoids integrating the given algebroid system. The existence of an inverse system of groupoids, however, does not by itself provide a strict inverse-limit Lie groupoid structure. We construct compatible local source charts from compatible algebroid trivializations and formulate a criterion ensuring that the levelwise integrations form a strict inverse system: the resulting compatible arrow-chart families are required to form a strict atlas, and the inverse-limit multiplication and inversion are required to be continuous with respect to the resulting atlas topology. Under these hypotheses, the inverse-limit groupoid integrates the inverse-limit algebroid.
Source-simply connectedness of this inverse-limit integration is a separate topological issue. For inverse sequences satisfying suitable local-triviality hypotheses on the source-fiber connecting maps, together with the required identification of the limit source fiber with the topological inverse limit of the level source fibers, we identify the obstruction in terms of the derived inverse limit
$\vpjl^{1}\pi_2(\G_{n,x_n},1_{x_n}).$
In particular, a Mittag-Leffler condition on the inverse sequence of second homotopy groups implies that the inverse-limit integration is source-simply connected.

Throughout, $I$ denotes a directed set.  An inverse system
is a contravariant functor from $I$ to the category under consideration. We write this as a pair
$(X_i,p_{ij})$.  If the connecting morphisms are understood, we
denote the inverse system by $\{X_i\}$.
We denote the limit maps by $p_j\colon \vpjl X_i\to X_j$. We also use this limit map notation for manifolds, bundles, and groupoids.
If $I=\N$, we interchangeably call the inverse system an inverse sequence.
All locally convex spaces are real and Hausdorff. 
Let $(E_i,\epsilon_{ij})$ be an inverse system of locally convex
spaces and let $(F_i, \epsij|_{F_j})$ be an inverse system of complemented 
subspaces, that is,
$F_i$ is a complemented subspace of $E_i$ and for all $i\leq j$, $\epsij(F_j)\subseteq F_i$. We say that 
$\{F_i\}$ is \Emph{complemented} (inverse system)
if there exists a family of continuous projections $\pi_i\colon E_i\to E_i$ with $\pi_i(E_i) = F_i$ such that 
$\pi_i\circ \epsij = \epsij\circ \pi_j$ whenever $i\leq j$. In turn,
$\vpjl F_i$ is complemented in $\vpjl E_i$ with $\vpjl \pi_i$
as the continuous projection. Equivalently, there exists an inverse system $(C_i,\nu_{ij})$ and compatible topological isomorphisms
$\theta_i\colon E_i\to F_i\times C_i$ such that 
$\theta_i(w) = (w,0)$ for all $w\in F_i$ and
$(\epsij|_{F_j}\times \nu_{ij})\circ \theta_j = \theta_i\circ \epsij$.
That is, $\vpjl \theta_i$ is a topological isomorphism.

When $I=\mathbb N$ and the model spaces are Banach spaces, their
inverse limit is a Fr\'echet space.  Consequently, the corresponding
strict inverse-limit manifolds are Fr\'echet manifolds. In the strong
case their topology is precisely the usual inverse-limit topology.

We use Bastiani
calculus; see \cite{GN26} for background on locally convex calculus,
manifolds, and Lie groups. For Banach spaces, Bastiani smoothness agrees with the usual
smoothness defined in terms of Fr\'echet derivatives.
If $E,F$ are locally convex spaces,
$U\subseteq E$ is open, and $f\colon U\to F$ is continuous, then $f$
is of class $C^1$ if
$df(x)v := \lim_{t\to0} t^{-1}({f(x+tv)-f(x)})$
exists for every $x\in U$ and $v\in E$, and defines a continuous map
$df\colon U\times E\to F$, $(x,v)\mapsto df(x)v$.  
Inductively, $f$ is $C^k$ if $df$ is $C^{k-1}$,
and $f$ is smooth if it is $C^k$ for every $k$.

Throughout $k\in \N\cup\{\infty\}$.
For $C^k$-manifolds, we only require the existence of a 
maximal atlas on a topological space where the modeling spaces
for the charts are locally convex spaces
and the overlap maps are $C^k$ in the Bastiani sense.
We don't require second-countability or paracompactness. In Section~\ref{sec:examples}, finite-dimensional manifolds are assumed Hausdorff and paracompact.
The topology is not required to be Hausdorff for the arrow manifold of Lie groupoids. In fact, most of the constructions, we give here work for non-Hausdorff manifolds. For the manifold structure on inverse limits, we start with the inverse limit 
as a set and use the atlas topology (see below).

Suppose that $M$ is a $C^k$-manifold. 
A subset $S \subseteq M$ is a \Emph{submanifold} if it is locally modeled 
on a closed subspace of the modeling space. That is, for every $x \in S$, there is a 
chart $\phi\colon U\to M_{\phi}$ about $x$ such that $\phi(U \cap S) = \phi(U) \cap S_\phi$ 
for a closed subspace $S_\phi \subseteq M_\phi$. If each 
$S_\phi$ is complemented in $M_\phi$, then $S$ is a 
\Emph{split submanifold}.
A $C^k$-map $f \colon M \to N$ is a \Emph{submersion} if, 
for every $x \in M$, there exist charts $\phi\colon U\to M_{\phi}$ about $x$ and $\psi\colon V\to N_{\psi}$ about 
$f(x)$ such that 
$M_{\phi}=N_{\psi}\times C$ for a locally convex space $C$, 
$f(U) \subseteq V$, and $\psi \circ f \circ \phi^{-1}=\pr_1|_{\phi(U)}$ 
for $\pr_1 \colon  N_\psi \times 
C \to N_\psi$. 
On the other hand, $f$ is an \Emph{immersion}  if instead
$N_\psi = M_\phi\times C$ and 
$\psi \circ f \circ \phi^{-1} = 
\iota|_{\phi(U)}$, where $\iota \colon M_\phi \to M_{\phi}\times C$
is the canonical inclusion.
In 
particular, an immersion restricts to an embedding locally about each point. 
An immersion that is also a topological embedding is called an 
\Emph{embedding}.

A groupoid $\G\rightrightarrows M$ with structure maps
$\src$, $\tgt$, $\mul$, $\one$, and $\inv$ is a
\Emph{(locally convex) Lie groupoid} if $M$ and $\G$ are smooth
manifolds, $\src$ is a submersion, $\inv$ is smooth, and
$\mul\colon\mset{\G}\to\G$ is smooth, where
$\mset{\G}:=(\src\times\tgt)^{-1}(\Delta_M)$ is endowed with its
split submanifold structure.
Here $\inv^2=\id_{\G}$, hence $\inv$ is a
diffeomorphism. In turn, $\tgt=\src\circ\inv$ is also a submersion.
Moreover, the unit map is automatically a smooth embedding.
Indeed, if $x\in M$, then a local smooth section
$\sigma\colon U\to\G$ of $\src$ through $1_x$. On $U$ we have
$\one=\mul\circ(\inv\circ\sigma,\sigma)$.
Thus $\one$ is smooth. Since $\src\circ\one=\id_M$, the unit map is
a smooth section of a submersion and hence an embedding.

Let $E$ be a locally convex space. The canonical projection
$\pr_1 \colon M \times E \to M$ defines the trivial vector bundle structure on
$M\times E$ with fiber $E$. Let $N \times F$ be another 
trivial $C^k$-vector bundle. A \Emph{vector bundle morphism} is a 
$C^k$-map $\Psi \colon M \times E \to N \times F$ which is 
given in terms of a pair of $C^k$-maps $f \colon M \to N$ and 
$K\colon M \times E \to F$ such that 
$\Psi(x, u) = (f(x), K(x, u))$ and $K$ is linear in the second variable.
Let $\Psi_i \colon M_i \times E_i \to M_{i+1} \times E_{i+1}$, 
$i = 1, 2$, be $C^k$-vector bundle morphisms. The composition is given by
\[
(\Psi_2 \circ \Psi_1)(x, u) 
= (f_2(f_1(x)), K_2(f_1(x), K_1(x, u)))
\]
which is a vector bundle morphism from $M_1\times E_1$ to $M_3\times E_3$.
We say that $\Psi$ is a \Emph{vector bundle isomorphism} if there exists a $C^k$
vector bundle morphism $\Psi^{-1} \colon N \times F \to 
M \times E$ which is inverse for $\Psi$.
In particular, 
$f$ is a $C^k$-diffeomorphism and $K(x,\cdot)$ is linear isomorphism 
for all $x\in M$.
We can write $\Psi^{-1}$  as $\Psi^{-1}(y, w) = (f^{-1}(y), K(f^{-1}(y),\cdot)^{-1} w) $.

A \Emph{$C^k$-vector bundle} is a $C^k$-manifold $A$ together with a 
$C^k$-surjective map $p \colon A \to M$  such that, 
for each $x \in M$, the fiber $A_x := p^{-1}(x)$ is a locally convex space, 
and such that there exists a family of $C^k$-diffeomorphisms (\Emph{local 
trivializations}) $\phi \colon p^{-1}(U) \to U \times F$ 
where $U$ is an open set (a trivializing neighborhood) and $F$ is a locally convex space such that the trivializing neighborhoods 
 form
an open cover of $M$ and satisfy that the following conditions:
\begin{enumerate}
    \item $\pr_2\circ \phi|_{A_x} \colon A_x \to F$ is a 
    topological linear isomorphism for all $x \in U$;
    \item $\pr_1 \circ \phi = p|_{p^{-1}(U)}$, where 
    $\pr_1 \colon U \times F \to U$ is the canonical projection;
    \item for any two local trivializations $\phi\colon p^{-1}(U)\to U\times F$ and $\psi\colon p^{-1}(V)\to V\times H$ with 
    $U\cap V \neq \emptyset$, the \Emph{transition map} 
    $\psi \circ \phi^{-1} \colon (U \cap V) \times F 
    \to (U\cap V) \times H$ is a $C^k$-isomorphism 
    between trivial vector bundles.
\end{enumerate}
Thus, we have a topological isomorphism $F\cong H$ so over each connected component we may speak of a single fiber.
For a subset $S \subseteq M$, we denote $A|_S := p^{-1}(S)$.
Similarly to the finite-dimensional case, the local triviality condition 
ensures that $p$ is a submersion.
One can alternatively define a 
vector bundle as a set and use the manifold structure defined by the system 
of local trivializations. We do this for the manifold structure on inverse limit.
In the smooth Banach category, the vector bundle definition we employ here 
is equivalent to the standard one~\cite{Sha26}.

Let $\wt{p} \colon \wt{A} \to \wt{M}$ be another vector 
bundle. A \Emph{vector bundle morphism} from $A \to \wt{A}$ is given 
by a pair of $C^k$ maps $(\Phi, f)$ where $\Phi \colon A\to \wt{A}$ 
and $f \colon M \to \wt{M}$ satisfy: $\wt{p} \circ \Phi = f \circ p$ 
and $\Phi_x:=\Phi|_{A_x} \colon A_x \to \wt{A}_{f(x)}$ is linear. 
In local trivializations, the map $\Phi$ is a $C^k$ morphism of trivial vector bundles.
Composition of 
$C^k$-vector bundle morphisms is a $C^k$-vector bundle morphism.

Let $f \colon N \to M$ be a $C^k$-map. The \Emph{pullback bundle} 
$f^*p\colon f^*A \to N$ is constructed in the usual way; we set $f^*A := N \times_{M} A$. 
This has a $C^k$ manifold structure for which the  projection $f^*p := \pr_1|_{f^*A} \colon 
f^*A \to N$ is a submersion.
Let $y \in N$ and let $\psi\colon p^{-1}(U)\to U\times F$ be a local 
trivialization about $f(y)$. Let $V := f^{-1}(U)$. Note 
that $(f^*p)^{-1}(V) = (f^*A) \cap (V 
\times p^{-1}(U))$. Construct a $C^k$ diffeomorphism 
$\wt{\psi} 
\colon (f^*p)^{-1}(V) \to V \times 
F$ by restricting $\id_{V} \times (\pr_2 \circ \psi)$. 
This is a local trivialization of $f^*A$ about $y$.
The inverse is given 
by $\wt{\psi}^{-1}(x, u) = (x, \psi^{-1}(f(x), u))$. For another local 
trivialization $\phi\colon p^{-1}(W)\to W\times H$ of $A$ 
defining  $\wt{\phi}\colon (f^*p)^{-1}(P) \to P\times H$ in the 
same manner with $P=f^{-1}(W)$, the overlap map $\wt{\phi} \circ \wt{\psi}^{-1}$ 
is given by the $C^k$ map $(x, u) \mapsto (x, \pr_2((\phi \circ \psi^{-1})(f(x), u)))$.
This shows that $f^*p \colon f^*A \to N$ is a $C^k$ vector bundle.

Finally, by a Lie algebroid we mean a first-order Lie
algebroid in the sense of \cite{Sha26}, and Lie algebroid morphisms are
understood in the differential-graded sense used there.  This
first-order condition is automatic in finite dimensions but is an
additional regularity condition in the locally convex case.

For a vector bundle $A\to M$, let
$\shA$ denote its sheaf of smooth local sections.

A \Emph{Lie algebroid}  is a smooth vector bundle $A\to M$ alongside a smooth vector bundle morphism
 $\anch\colon A\to TM$ over $\id_M$ called the \Emph{anchor} together with a bracket of
sheaves of $\R$-modules
$[\cdot,\cdot]_A\colon\shA\otimes_\R\shA\to\shA$ such that:
\begin{enumerate}\itemsep0em
\item $\shA(U)$ is a Lie algebra for every open $U\subseteq M$;
\item $[\xi,f\eta]_A
=f[\xi,\eta]_A+(\anch(\xi).f)\eta$;
\item in every local trivialization
$\phi\colon A|_U\to U\times F$, there is a smooth map
$C_\phi\colon U\times F\times F\to F$,
alternating bilinear in the last two variables, such that
\begin{equation}\label{eq:C-LA}
[\xi,\eta]^\phi(x)
=
\anch(\xi).\eta^\phi(x)
-\anch(\eta).\xi^\phi(x)
+C_\phi(x,\xi^\phi(x),\eta^\phi(x)).
\end{equation}
Here $\xi^\phi$ are the local representatives of $\xi$ in $\phi$ and similarly for $\eta,[\xi,\eta]$.
\end{enumerate}

The formula~\eqref{eq:C-LA} allows the bracket on an inverse limit to be
defined on arbitrary smooth local sections rather than only on
sections descending from the levels; see Theorem~\ref{thm:invlimit-algebroid} below. It also gives the differential
calculus on scalar-valued forms avoiding the dual bundle. 
Let $\Omega_A^q$ denote the sheaf of smooth
fiberwise continuous alternating $q$-forms, with
$\Omega_A^0=\shcinf_M$. That is, for any $U\subseteq M$ open, we assign
\[
\Omega_A^q(U)
=
\{\omega\in C^\infty(A|_U^{\times_U q},\mathbb R):
\omega_x\text{ is alternating \(q\)-linear for all }x\in U\}.
\]
The  differential of $\omega\in \Omega^q_A(U)$ is defined by
\begin{equation}\label{eq:d_A}
\begin{aligned}
(d_A\omega)_x(a_0,\ldots,a_q)
&=
\sum_{i=0}^{q}(-1)^i
\anch(\xi_i).
\bigl(
\omega(\xi_0,\ldots,\widehat{\xi_i},\ldots,\xi_q)
\bigr)(x)
\\&\quad +
\sum_{i<j}(-1)^{i+j}
\omega\bigl(
[\xi_i,\xi_j]_A,
\xi_0,\ldots,\widehat{\xi_i},\ldots,
\widehat{\xi_j},\ldots,\xi_q \bigr)(x),
\end{aligned}
\end{equation}
where $x\in U$, $a_i\in A_x$, and $\xi_i\in\shA(V)$ are arbitrary
smooth local extensions on some neighborhood $V\subseteq U$ of $x$
satisfying $\xi_i(x)=a_i$.
This is shown to be a well-defined $(q+1)$-form by passing to a local trivialization
and using \eqref{eq:C-LA}.
In fact, $d_A^2=0$. Moreover, one can show that $\Omega_A^\bullet:=\bigoplus_{q\geq 0}\Omega^q_A$ is a sheaf of dg-algebras;
see \cite{Sha26}.

Let $A \to M$ and $\wt{A} \to \wt{M}$ be first-order Lie algebroids.
Let $\Phi \colon A \to \wt{A}$ be a smooth vector bundle morphism over
$f\colon M\to \wt M$.
For $\omega\in \Omega^q_{\wt A}(\wt U)$,  we define
 $\Phi^*\omega\in \Omega^q_{A}(f^{-1}(\wt U)$ 
 by
 $(\Phi^* \omega)_x(a_1, \dots, a_q) = \omega_{f(x)}(\Phi_x a_1, \dots, \Phi_x a_q),$
 for $x\in f^{-1}(\wt U)$ and $a_i\in A_x$.
 We say that $\Phi$ is a \Emph{Lie algebroid morphism} if 
 $d_A(\Phi^* \omega) = \Phi^*(d_{\wt{A}} \omega)$ for every such $\omega$.
More precisely, $\Phi$
is a Lie algebroid morphism if the induced pullback sheaf morphism
$\Phi^\sharp \colon f^{-1}\Omega^\bullet_{\wt{A}} \to \Omega^\bullet_A$
commutes with the Lie algebroid differentials:
$\Phi^\sharp \circ f^{-1}d_{\wt{A}} = d_A \circ \Phi^\sharp$; See~\cite{Sha26} for the details.

For a Lie groupoid $\G\rightrightarrows M$, its Lie algebroid is the
vector bundle $\Lie(\G):=\ker(T\src)|_{\one(M)}\to M$. Its anchor is
the restriction $\anch:=T\tgt|_{\Lie(\G)}$, and its bracket is
determined by the identification of sections of $\Lie(\G)$ with
right-invariant $\src$-vertical vector fields on $\G$. This is a Lie
algebroid in the sense defined above. If
$(\F,f)\colon\G\rightrightarrows M\to\wt\G\rightrightarrows \wt M$ is a Lie
groupoid morphism, then
$\Lie(\F):=T\F|_{\Lie(\G)}\colon\Lie(\G)\to\Lie(\wt \G)$ is a Lie algebroid
morphism over $f$. In this way, $\Lie$ defines a functor from the
category of locally convex Lie groupoids to the category of Lie
algebroids defined above; see~\cite{Sha26} for the proofs in the
locally convex setting.

\medskip
The paper is organized as follows.  Section~\ref{sec:mfds_bundles} develops
strict inverse limits of manifolds and vector bundles.
Section~\ref{sec:algd-gpd} treats Lie algebroids and Lie groupoids and
proves the compatibility of the Lie functor with inverse limits.
Section~\ref{sec:integration} studies inverse systems of integrable
Banach-Lie algebroids and the source-simply connectedness obstruction.
Section~\ref{sec:examples} concludes with examples illustrating the
strict/strong distinction and the integration results.

\section{Strict inverse limits of manifolds}\label{sec:mfds_bundles}
We start with the following elementary lemma. 
\begin{lemma}\label{lemma:basic-C^k}
    Let $(E_i,\epsij)$ and $(F_i,\nu_{ij})$ be inverse systems
    of locally convex spaces. Let $f_i\colon U_i\to F_i$ be
    a compatible family of $C^k$ maps such that each $U_i\subseteq E_i$
    is open, $\epsij(U_j)\subseteq U_i$ whenever $i\leq j$, and
    $\vpjl U_i$ is open in $\vpjl E_i$. Then $\vpjl f_i$ is $C^k$.

    Furthermore, let $(H_i,\tau_{ij})$ be an inverse system of locally
    convex spaces, and let $g_i\colon V_i\to H_i$ be another compatible
    family of $C^k$ maps such that each $V_i\subseteq F_i$ is open,
    $f_i(U_i)\subseteq V_i$, $\nu_{ij}(V_j)\subseteq V_i$ whenever
    $i\leq j$, and $\vpjl V_i$ is open in $\vpjl F_i$.
    Then $\vpjl(g_i\circ f_i)=\vpjl g_i\circ\vpjl f_i$.
\end{lemma}

\begin{proof}
    Continuity follows from the universal property of the
    inverse-limit topology. If $k\geq1$, differentiating
    $\nu_{ij}\circ f_j=f_i\circ\epsij$ and using the chain rule show that
    $(df_i(x_i,v_i))_i\in\vpjl F_i$ for any $(x_i)\in\vpjl U_i$, 
    $(v_i)\in \vpjl E_i$. Since convergence in
    $\vpjl F_i$ is componentwise, the derivative of $\vpjl f_i$ is
    $d(\vpjl f_i)((x_i),(v_i))=(df_i(x_i,v_i))_i$. 
    The composition $\nu_j\circ d(\vpjl f_i)$ is continuous,
    hence $d(\vpjl f_i)$ is continuous. Thus $\vpjl f_i$ is $C^1$, and induction
    gives the assertion for $C^k$. The final identity follows
    componentwise.
\end{proof}

We now define strict inverse systems of manifolds and bundles
and record some elementary facts without proofs. The proofs
follow immediately by passing to charts in most cases.

Let $(M_i,g_{ij})$ be an inverse system of $C^k$ manifolds. A
\Emph{strict inverse system chart} (\sis chart for short) about
$(x_i)\in\vpjl M_i$ is a family of charts
$\phi_i\colon U_i\to M_{\phi_i}$ about $x_i$ such that:
\begin{enumerate}
    \item the model spaces form an inverse system of locally convex spaces
    $(M_{\phi_i},\epsilon_{ij})$;

    \item $g_{ij}(U_j)\subseteq U_i$ whenever $i\leq j$;

    \item $\epsilon_{ij}\circ\phi_j=\phi_i\circ g_{ij}$ on $U_j$
    whenever $i\leq j$;

    \item $\vpjl \phi_i(U_i)$ is open in $\vpjl M_{\phi_i}$
    for the inverse-limit topology.
\end{enumerate}
A \Emph{strict atlas} is a family $\cA$ of \sis charts covering $\vpjl M_i$ such that for every $\{\phi_i\},\{\wt\phi_i\}\in \cA$, 
the overlaps $\vpjl \phi_i(U_i\cap\wt U_i) $ and $\vpjl \wt\phi_i(U_i\cap \wt U_i)$ are open. By Lemma~\ref{lemma:basic-C^k}, it follows that the overlap maps $\big(\vpjl \wt\phi_i\big)\circ\big(\vpjl \phi_i\big)^{-1}$ are $C^k$.
An inverse system of manifolds is \Emph{strict} if it 
is equipped with
a strict atlas $\cA$. Whenever we say an inverse system is strict, 
we mean that $\cA$ has been selected. 
We use it to define a $C^k$ manifold structure on $\vpjl M_i$ (with the atlas topology). It follows that $g_j\colon \vpjl M_i\to M_j$ are continuous by the universal property of the atlas topology as the final topology. That is, the atlas topology can be finer than the inverse limit topology on $\vpjl M_i$. Hence the atlas topology
is Hausdorff if each $M_i$ is Hausdorff.
A strict atlas is \Emph{\strong} if the atlas topology agrees with the
inverse limit topology.
That is, $\cA$ is strong if and only if $\vpjl U_i$ is open for each \sis chart $\{\phi_i\colon U_i\to M_{\phi_i}\}\in \cA$.
A strict inverse system with a strong atlas is called strong in turn.
If $\cA$ is strict, then $T\cA:=\{ \{T\phi_i\} : \{\phi_i\}\in \cA\}$ is a strict atlas for $(TM_i, Tg_{ij}) $. If $\cA$ is strong, then $T\cA$ is also strong.  Moreover, one can show that $T\vpjl M_i\cong \vpjl TM_i $. Locally, this is the diffeomorphism
$\vpjl \phi_i(U_i)\times\vpjl M_{\phi_i}\cong \vpjl (\phi_i(U_i)\times M_{\phi_i})$ which maps $((y_i),(v_i))$ to $(y_i,v_i)_i$.

The dyadic solenoid provides an example of a strict
inverse system which is not strong. Consider the
inverse sequence $(S_n,p_{nm})$, where $S_n=S^1$
and $p_{nm}(z)=z^{2^{m-n}}$ whenever $n\leq m$.
Its inverse limit $\vpjl S_n$ is the dyadic
solenoid. Define the injective group homomorphism
$h:\R\to\vpjl S_n$ by
$h(t)=(e^{it/2^n})_n$.
For $x=(x_n)\in\vpjl S_n$ and $0<\delta<\pi$,
define $U_n^x:=\{x_ne^{it/2^n}:|t|<\delta\}$
and $\phi_n^x(x_ne^{it/2^n})=t/2^n$.
These charts are compatible with the connecting
maps $p_{nm}$ and the linear maps
$u\mapsto 2^{m-n}u$. Their inverse-limit chart
domain is $xh((-\delta,\delta))$, and their
inverse-limit chart image identifies with the
open interval $(-\delta,\delta)\subseteq\R$.
If two such domains intersect, their centers
satisfy $y=xh(c)$ for some $c\in\R$, and the
corresponding overlap map is the restriction
of the translation $t\mapsto t-c$.
Thus these chart families define a strict
smooth atlas on $\vpjl S_n$.
This atlas is not strong. Indeed, the sequence
$a^{(N)}:=h(2\pi 2^N)$ converges to the identity
in the inverse limit topology, since
its $n$-th coordinate equals $1$ whenever
$n\leq N$. However, none of its terms belongs
to $h((-\delta,\delta))$, by injectivity of $h$.
Hence this chart domain is not open in the
inverse-limit topology.
Moreover, multiplication and inversion are
smooth in the strict atlas, since they are
locally represented by addition and negation
in compatible coordinates. Thus $\vpjl S_n$
is a locally convex Lie group. Finally,
$h(\R)$ is an open proper subgroup of
$\vpjl S_n$. Indeed, it is open by the chart
construction and is torsion-free, whereas
$\vpjl S_n$ contains an element of order three:
take $a=(a_n)$ with $a_n=\zeta$ for odd $n$
and $a_n=\zeta^{-1}$ for even $n$, where
$\zeta=e^{2\pi i/3}$. Consequently,
$\vpjl S_n$ is disconnected in its atlas
topology.

A \Emph{morphism of strict inverse systems of manifolds} from $(M_i,g_{ij})$ to $(N_i,h_{ij})$ is a family of compatible $C^k$ maps such that $\vpjl f_i$ is continuous with respect to the atlas topologies. In fact, $\vpjl f_i$ is automatically continuous if $\{N_i\}$ is strong. Locally using Lemma~\ref{lemma:basic-C^k}, one can show that $\vpjl f_i$ is $C^k$. Strict inverse systems of manifolds
with these morphisms 
form a category. The canonical projections $\{\pi_{TM_i}\}$ define a morphism of strict $C^{k-1}$ inverse systems.
For $k=\infty$ and compatible families of smooth vector fields $\{X_i\}, \{Y_i\}$ which are morphisms of strict inverse systems, one can show that $[\vpjl X_i,\vpjl Y_i] = \vpjl [X_i,Y_i]$ computed using derivatives of representatives in \sis charts (suppressing the identification $T\vpjl M_i\cong \vpjl TM_i$). 

A \Emph{strict inverse system of submanifolds} of $(M_i,g_{ij},\cA)$ 
is a family $\{S_i\}$ such that 
\begin{enumerate}
    \item $S_i\subseteq M_i$ is a $C^k$ submanifold;
    \item $g_{ij}(S_j)\subseteq S_i$;
    \item for every $(x_i)\in\vpjl S_i$, there is a \sis chart $\{\phi_i\}\in \cA$ such that $\phi_i(S_i\cap U_i) = \phi_i(U_i)\cap S_{\phi_i}$ where $S_{\phi_i}\subseteq M_{\phi_i}$ is a closed subspace.
\end{enumerate}
$\{\phi_i\}$ is called \Emph{\sis submanifold chart}. 
Suppose that $\{\phi_i\}$ is centered at $(x_i)\in\vpjl S_i$. For $i\leq j$, one computes $\epsij(\phi_j(U_j)\cap S_{\phi_j})\subseteq \phi_i(U_i)\cap S_{\phi_i}$. Now $\phi_j(U_j)\cap S_{\phi_j}$ is a neighborhood of $0\in S_{\phi_j}$ and $\epsij$ is linear. Hence, $\epsij(S_{\phi_j})\subseteq S_{\phi_i}$. Thus, $(S_{\phi_i}, \epsij)$
is an inverse system of locally convex spaces. In particular, $\vpjl \phi_i(\vpjl U_i\cap \vpjl S_i) = \vpjl \phi_i(U_i)\cap \vpjl S_{\phi_i}$. 
We say that $\{S_i\}$ is \Emph{split} if the subsystem $\{S_{\phi_i}\}$ is complemented. 
Then, $\vpjl S_{\phi_i}$ is complemented in $\vpjl M_{\phi_i}$. 
Thus, $\vpjl S_i$ is a $C^k$ submanifold of $\vpjl M_i$, and if $\{S_i\}$
is split, then $\vpjl S_i$ is a split submanifold. 
Clearly, the family $\{\Delta_{M_i}\}$ is a split strict inverse system of submanifolds of $(M_i\times M_i, g_{ij}\times g_{ij})$.
If $(M_i, g_{ij})$ is strong and $\{S_i\}$ is  a strict inverse system of submanifolds, then $\{S_i\}$, equipped with the atlas induced by the \sis submanifolds charts, is strong.

Let $\{f_i\colon M_i\to N_i\}$ be a  morphism of strict $C^k$ inverse systems of manifolds such that each $f_i$ is a submersion. We say that 
$\{f_i\}$ is a \Emph{split-compatible family of submersions} if for every $(x_i)\in \vpjl M_i$,
there is a \sis chart $\{\phi_i\}$ about $(x_i)$ and $\{\psi_i\}$ about
$(f_i(x_i))$ such that
$ \{N_{\psi_i}\}$ is complemented in $\{M_{\phi_i}\}$ and 
$\psi_i\circ f_i\circ \phi_i^{-1} $
is the restriction of $\pr_1\colon M_{\phi_i}=N_{\psi_i}\times C_i \to N_{\psi_i}$. 
\begin{proposition}\label{prop:submersions}
Let $(M_i,g_{ij})$ and $(N_i,h_{ij})$ be strict inverse systems of manifolds.
Let $\{f_i\colon M_i\to N_i\}$ be a split-compatible 
family of submersions. Then:
\begin{enumerate}
    \item $\vpjl f_i\colon \vpjl M_i\to \vpjl N_i$ is a $C^k$ submersion; \label{item:sub-prop-item1}
    \item if $\{S_i\}$ is (split) strict inverse system of submanifolds of $\{N_i\}$, then $\{f_i^{-1}(S_i)\}$ 
    is a (split) strict inverse system of submanifolds of $\{M_i\}$. Moreover, $\big(\vpjl f_i \big)^{-1}(\vpjl S_i) = \vpjl (f_i^{-1}(S_i))$; \label{item:sub-prop-item2}
    \item for every $(x_i)\in\vpjl M_i$, there exists a compatible family of local sections $s_i$ of $f_i$ through $x_i$ for which $\vpjl s_i$
    is a local section of $\vpjl f_i$ through $(x_i)$;
    \label{item:sub-prop-item3}
    \item if $(\wt M_i, \wt g_{ij})$ is another strict inverse system of manifolds with a morphism of strict inverse systems of manifolds
    $\{\wt f_i\colon \wt M_i\to N_i\}$, then 
    \begin{enumerate}
        \item $\{M_i\times_{N_i} \wt M_i\}$ is a split strict inverse 
    system of submanifolds of $\{M_i\times \wt M_i\}$;
    \item  the projections $\pr_2\colon M_i\times_{N_i}\wt M_i\to \wt M_i$ form a split-compatible family of submersions;
    \item     $
    (\vpjl M_i)\times_{\vpjl N_i}(\vpjl \wt M_i) \cong \vpjl (M_i\times_{N_i}\wt M_i).
    $
    \end{enumerate}\label{item:sub-prop-item4}
\end{enumerate}
\end{proposition}
\begin{proof}
                Take
    $\{\phi_i\colon U_i\to M_{\phi_i}\}$ and 
    $\{\psi_i\colon V_i\to N_{\psi_i}\}$ to be \sis charts
    such that $f_i(U_i)\subseteq V_i$,
     $M_{\phi_i}=N_{\psi_i}\times C_i$, and $\psi_i\circ f_i\circ \phi_i^{-1}(u,c) = u$, where $\{C_i\}$ is a complemented
    inverse system. Passing to the limit shows that $\vpjl f_i$ is a submersion. This shows~\ref{item:sub-prop-item1}.

    Let $(x_i)\in \vpjl f_i^{-1}(S_i)$. If $\{\psi_i\}$ is a \sis submanifold chart 
    for $\{S_i\}$
    about $(f_i(x_i))$ 
    such that $\{S_i\}$ is modeled on $\{H_i\}$, then we can choose 
    a \sis chart $\{\phi_i\}$
    which alongside $\{\psi_i\}$ provide split-compatible
    submersion charts for $\{f_i\}$.
    Indeed, let 
    $\{\chi_i\colon W_i\to N_{\kappa_i}\times C_i\}$ and $\{\kappa_i\colon P_i\to N_{\kappa_i}\}$ be split-compatible submersion charts
    for $\{f_i\}$ about $(x_i)$ and $(f_i(x_i))$, respectively.
        Set $\phi_i(y) = (\psi_i(f_i(y)), \pr_{C_i}(\chi_i(y)))$
    over $U_i:=W_i\cap f_i^{-1}(V_i\cap P_i)$. Clearly $\vpjl U_i$
    is open in the atlas topology. Hence $\{\phi_i\}$ is an
    admissible \sis chart for $\{M_i\}$.
 By construction, it is a \sis submanifold chart for $\{f_i^{-1}(S_i)\}$
    which is modeled on $\{H_i\times C_i\}$.
    If $\{S_i\}$ is split, then $\{H_i\times C_i\}$ is complemented and consequently 
    $\{f_i^{-1}(S_i)\}$ is split.
    The equality is obvious. This shows~\ref{item:sub-prop-item2}.

    For~\ref{item:sub-prop-item3}, let $(x_i) \in \vpjl U_i$ and write $\phi_i(x_i) = (y_i,c_i)$. The sets
    $O_i :=\psi_i^{-1}\big(\{ u \in N_{\psi_i}: (u,c_i)\in \phi_i(U_i)\}\big)$
    are open and satisfy $h_{ij}(O_j)\subseteq O_i$.
    On $O_i$, define $s_i(z) = \phi_i^{-1}(\psi_i(z),c_i)$. Then $\{s_i\}$
    are compatible and each $s_i$ is a local section of $f_i$ through $x_i$.
    Finally, $\vpjl s_i$ is smooth on $\vpjl O_i$ which is open.
    Indeed, $(\vpjl \psi_i)(\vpjl O_i)$ is the
    preimage of $\vpjl \phi_i(U_i)$ under $u\mapsto (u, (c_i))$ hence open.

    Take  $\{\theta_i\colon W_i\to \wt M_{\theta_i}\}$ to be a \sis chart such that $\wt f_i(W_i)\subseteq V_i$. Let $r_i$ represent $\wt f_i$ in these charts. Then
    the maps $(u,c,z)\mapsto (u-r_i(z),c,z)$
    are  diffeomorphisms onto their open images and compatible with the connecting maps. Hence they define a \sis submanifold chart for $M_i\times_{N_i}\wt M_i$ 
    on the complemented subspace $\{0\}\times C_i\times \wt M_{\theta_i}$.
        The three assertions in~\ref{item:sub-prop-item4} then
    follow immediately using these charts.
\end{proof}
In the same way, $\{f_i\}$ is a \Emph{split-compatible family of immersions} if instead 
$\{M_{\phi_i}\}$ is complemented in $\{N_{\psi_i}\}$ and
$\psi_i\circ f_i\circ \phi_i^{-1} $ is the restriction of the canonical embedding $M_{\phi_i}\to M_{\phi_i}\times C_i = N_{\psi_i}$.
One can see the following facts:
\begin{proposition}\label{prop:immersions}
    Let $\{f_i\colon M_i\to N_i\}$ be a split-compatible family of immersions. Then:
\begin{enumerate}
    \item $\vpjl f_i$ is an immersion;
    \item if each $f_i$ is an embedding and $\{N_i\}$ is strong, then 
    $\vpjl f_i$ is an embedding if and only if $\{M_i\}$ is strong.\label{item:embedding-strong}
\end{enumerate}
\end{proposition}
In the second assertion, even if  $\{N_i\}$ is strict, $\vpjl f_i$ is
 an embedding
if  $\{M_i\}$ is strong. Indeed, the topology induced 
on $(\vpjl f_i)(\vpjl M_i)$ is coarser than the atlas topology
on $\vpjl M_i$, but
finer than the inverse-limit topology. If $\{M_i\}$ is strong, these topologies agree.

An inverse system of vector bundles $(p_i\colon A_i\to M_i, \Phi_{ij},f_{ij})$
is  \Emph{strict} if:
\begin{enumerate}
    \item $(M_i,f_{ij})$ is a strict inverse system of manifolds;
    \item for every $(x_i)\in \vpjl M_i$, there exist vector bundle trivializations $\{\psi_i\colon p_i^{-1}(U_i)\to U_i\times F_{i}\} $ alongside an inverse system of locally convex spaces $(F_{i}, \nu_{ij})$ such that
    $f_{ij}(U_j)\subseteq U_i$ and $\psi_i\circ \Phi_{ij} = (f_{ij}\times \nu_{ij})\circ \psi_j$ whenever $i\leq j$, and such that
    $\vpjl U_i$ is an open 
        neighborhood of $(x_i)$ (in the atlas topology).
\end{enumerate}
Such a family $\{\psi_i\}$ is called \Emph{\sis local trivialization}.
The family of
\sis local trivializations and the strict atlas on $\{M_i\}$
make
$(A_i,\Phi_{ij})$ a strict inverse system of manifolds. Moreover, $\vpjl p_i\colon \vpjl A_i\to \vpjl M_i$ is a vector bundle where 
locally the fiber is $\vpjl F_i$ in the trivializing neighborhood $\vpjl U_i$ of $\vpjl \psi_i$.
Immediately from the definition one sees that $\{A_i\}$ is strong if and only if $\{M_i\}$ is strong.
A family $\{B_i\}$ is a \Emph{strict inverse system of subbundles}
of $(p_i\colon A_i\to M_i, \Phi_{ij},f_{ij})$
if \begin{enumerate}
    \item $B_i\subseteq A_i$ is a subbundle;
    \item $\Phi_{ij}(B_j)\subseteq B_i$;
    \item for every $(x_i)\in\vpjl M_i$, there is a \sis local trivialization $\{\psi_i\}$ such that $\psi_i(p_i^{-1}(U_i)\cap B_i) = U_i\times H_i $ where $H_i$ is a closed subspace of $F_i$.
\end{enumerate}
The system is \Emph{split} if we can choose $\{\psi_i\}$ so that $\{H_i\}$ is a complemented subsystem in $\{F_i\}$. It follows that 
$\vpjl B_i$ is a subbundle of $\vpjl A_i$. It is a split-subbundle if $\{B_i\}$ is split.

The following lemma is immediate.
\begin{lemma}\label{lemma:pullback-sum-bundles}
Let $(p_i\colon A_i\to M_i, \Phi_{ij},f_{ij})$ be a strict inverse system
of vector bundles. 
\begin{enumerate}
    \item If $(\wt A_i\to M_i, \wt \Phi_{ij},f_{ij})$ is another strict inverse system, then $(A_i\oplus \wt A_i\to M_i, \Phi_{ij}\times \wt \Phi_{ij}, f_{ij})$ is a strict inverse system of
    vector bundles. Moreover, $\vpjl (A_i\oplus \wt A_i)\cong 
    \vpjl A_i\oplus \vpjl \wt A_i$.
    \item If $(N_i,g_{ij})$ is a strict inverse system of manifolds and $\{h_i\colon N_i\to M_i\}$ is a morphism of strict inverse systems, then $(h_i^*A_i\to N_i, g_{ij}\times \Phi_{ij}, g_{ij})$
    is a strict inverse system of vector bundles and $\vpjl (h_i^*A_i)\cong (\vpjl h_i)^*(\vpjl A_i)$ is a canonical vector bundle isomorphism.
\end{enumerate}

\end{lemma}

Let $(\wt p_i\colon \wt A_i\to \wt M_i, \wt \Phi_{ij},\wt f_{ij})$
be another strict inverse system of vector bundles. A \Emph{morphism of strict inverse systems of vector bundles} is a compatible family $(\Theta_i,\theta_i)\colon A_i\to \wt A_i$ of vector bundle maps such that $\{\theta_i\}$ is a morphism of strict inverse systems of manifolds.
It follows that $\{(\Theta_i,\theta_i)\}$ induces a vector bundle map $(\vpjl \Theta_i,\vpjl \theta_i)\colon \vpjl A_i \to \vpjl \wt A_i$.
We say that $\{(\Theta_i,\theta_i)\}$ is \Emph{split-compatible}
if about every $(x_i)\in \vpjl M_i$, there exist \sis local trivializations $\{\psi_i\colon p_i^{-1}(U_{i})\to U_{i}\times F_{i}\}$ and $\{\wt \psi_i\colon \wt p_i^{-1}(\wt U_{i})\to \wt U_{i}\times \wt F_{i}\}$  of the source and the target bundles with $\theta_i(U_{i})\subseteq \wt U_{i}$ and continuous linear maps $T_i\colon F_i\to \wt F_i$ such that $(\wt\psi_i\circ \Theta_i\circ \psi_i^{-1})(x,v) = (\theta_i(x), T_iv)$ and $\{\ker T_i\}$ is a complemented subsystem in $\{F_i\}$.
\begin{lemma}\label{lemma:kernel-subbundle}
   Let $\{(\Theta_i,\theta_i)\}$ be a split-compatible morphism of
    strict inverse systems of vector bundles. Then $\{\ker(\Theta_i)\}$
    is a split strict inverse system of subbundles and $\ker(\vpjl (\Theta_i)) =\vpjl \ker(\Theta_i)$. 
    In particular, if
    $ \{f_i\colon M_i\to N_i\}$ is a split-compatible family of 
    submersions, then $\{\ker(Tf_i)\}$ is a split strict inverse system
    of subbundles of $\{TM_i\to M_i\}$ and $\ker(T\vpjl f_i)\cong \vpjl \ker(Tf_i)$.
\end{lemma}

The following examples and  remark illustrate the role of
split-compatibility in the previous results.

\begin{example}
\begin{enumerate}
\item \emph{The inverse limit of submersions need
not be a submersion.}
Let $I=\N\cup\{0\}$, $M_n=C^\infty(\R)$ and
$N_n=\R^{n+1}$. Let $g_{nm}=\id$ and let
$h_{nm}:\R^{m+1}\to\R^{n+1}$ be the
canonical projection whenever $n\leq m$.
Define $J_n\colon M_n\to N_n$ by
$J_n(f)=(f(0),\ldots,f^{(n)}(0))$ and put
$C_n:=\ker J_n$. Choose
$\chi\in C_c^\infty(\R)$ with $\chi=1$
near $0$, and define 
$c_n\colon N_n\to M_n$ by
$c_n(a)(x)=\chi(x)\sum_{j=0}^n a_jx^j/j!.$
Since $J_n\circ c_n=\id_{N_n}$, we have
$M_n=c_n(N_n)\oplus C_n$, and each $J_n$
is a smooth submersion.

The limit map is the full jet map
$J_\infty\colon C^\infty(\R)\to\R^{\N_0}$,
which is surjective by Borel's theorem~\cite{Bor95}.
However, $J_\infty$ is not a submersion.
Indeed, suppose that it admitted a
continuous linear right inverse $T$.
Then $q(a):=\sup_{|x|\leq 1}|T(a)(x)|$
is a continuous seminorm on $\R^{\N_0}$.
Hence there exist $r\in\N_0$ and $C>0$
such that
$q(a)\leq C\max_{0\leq j\leq r}|a_j|$.
For $m>r$, this gives $q(e_m)=0$, where
$e_m$ is the $m$-th standard basis vector.
Thus $T(e_m)$ vanishes on $[-1,1]$,
contradicting $J_\infty(T(e_m))=e_m$.

In particular, $\ker J_\infty=\vpjl C_n$
is not complemented in $C^\infty(\R)$.
Indeed, a closed complement would give
a continuous linear right inverse of
$J_\infty$ by the open-mapping theorem.
Moreover, although every $C_n$ is
complemented, the projections onto
$C_n$ cannot form a compatible family:
since $g_{nm}=\id$, compatibility would
force these projections to coincide,
whereas the subspaces $C_n$ are distinct.
Thus the family $\{J_n\}$ is not
split-compatible, although each $J_n$
is a submersion. Both inverse systems
are strong.
\item \emph{The limit map may be a submersion
without split-compatibility.}
Let $I=\N$, $M_n=\R^3$ and $N_n=\R^2$,
and define $f_n\colon M_n\to N_n$ by
$f_n(x,y,z)=(x,z)$. Let the connecting
maps be given by
$g_{n,n+1}(x,y,z)=(x,z,0)$ and
$h_{n,n+1}(x,z)=(x,0)$, with the
remaining maps defined by composition.
Then $f_n\circ g_{nm}=h_{nm}\circ f_m$
whenever $n\leq m$.

Each $f_n$ is a submersion. Moreover,
$\vpjl M_n=\{((x,0,0))_n:x\in\R\}$
and $\vpjl N_n=\{((x,0))_n:x\in\R\}$.
Hence $\vpjl f_n$ is a linear
isomorphism and thus a submersion.

Nevertheless, the kernels
$K_n:=\ker f_n=\{0\}\times\R\times\{0\}$
do not form a complemented inverse
subsystem. Indeed, suppose that
$P_n\colon M_n\to M_n$ were compatible
continuous linear projections onto
$K_n$. Since $g_{n,n+1}(K_{n+1})=\{0\}$,
compatibility would give
$P_n\circ g_{n,n+1}
=g_{n,n+1}\circ P_{n+1}=0$.
But $(g_{n,n+1})(\R^3)=\R^2\times\{0\}$
contains $K_n$, contradicting
$P_n|_{K_n}=\id_{K_n}$.
Thus no such compatible projections
exist. In particular, the family
$\{f_n\}$ is not split-compatible,
since differentiating split-compatible
charts at the zero sections would
give compatible projections onto
the kernels.

\item \emph{The inverse limit of
non-submersions may be a submersion.}
Let $E$ and $F$ be nonzero locally
convex spaces. Put $M_n=E\times F$
with identity connecting maps and
$N_n=F$ with zero connecting maps
between distinct levels. Let
$f_n\colon M_n\to N_n$ be the zero map.
None of the $f_n$ is a submersion,
whereas $\vpjl N_n=\{0\}$.
Consequently, $\vpjl f_n\colon  E\times F\to\{0\}$ is a submersion.
\end{enumerate}
\end{example}

\begin{remark}
The first example also shows that an
inverse limit of smooth split embeddings
need not be an immersion, even when both
inverse systems are strong. Indeed, each
inclusion $\iota_n\colon C_n\hookrightarrow
C^\infty(\R)$ is a smooth split embedding,
since $C_n$ is complemented in
$C^\infty(\R)$. Equip $\{C_n\}$ with
the canonical inclusions as the connecting maps and
$C^\infty(\R)$ with identity connecting
maps. The limit map is the inclusion
$\vpjl C_n=\ker J_\infty
\hookrightarrow C^\infty(\R)$,
which is a topological embedding but
not an immersion. Indeed, its differential
has noncomplemented image, whereas the
differential of an immersion has
complemented image.
\end{remark}

An inverse system of Lie groups $(G_i,h_{ij})$ is \Emph{strict} if it
is strict as an inverse system of manifolds and the families of
multiplication and inversion maps are morphisms of strict inverse
systems. As a result, $\vpjl G_i$ is a Lie group and 
the limit maps are Lie group morphisms.

 Let
$(\pi_i\colon P_i\to M_i, G_i, r_{ij}, f_{ij},h_{ij})$ be an
inverse system of principal bundles. That is, 
$\pi_i\colon P_i\to M_i$ are principal right $G_i$-bundles, 
$(P_i,r_{ij})$ is inverse system of manifolds,
 $(M_i,f_{ij})$ is an inverse system of manifolds, and
$(G_i,h_{ij})$ is an inverse system of Lie groups
such that
 $ \pi_i\circ r_{ij}=f_{ij}\circ\pi_j$ and 
    $r_{ij}(p\cdot g)=r_{ij}(p)\cdot h_{ij}(g).$
    Suppose that $\{M_i\}$ and $\{G_i\}$ are strict.
A family of principal trivializations
$\{\tau_i\colon\pi_i^{-1}(U_i)\to U_i\times G_i\}$ is called an
\Emph{\sis principal trivialization} if
$f_{ij}(U_j)\subseteq U_i$,
$\tau_i\circ r_{ij}=(f_{ij}\times h_{ij})\circ\tau_j$, and
$\vpjl U_i$ is open in $\vpjl M_i$. Write
$\wt\tau_i:=\pr_2\circ\tau_i$.
If  $\{\lambda_i\colon \pi_i^{-1}(V_i)\to V_i\times G_i\}$ is another \sis principal
trivializations, set
$c_i^{\lambda\tau}(x):=
\wt\lambda_i(\tau_i^{-1}(x,e_i))$. Then
$\wt\lambda_i(p)
=c_i^{\lambda\tau}(\pi_i(p))\wt\tau_i(p)$ and
    $h_{ij}\circ c_j^{\lambda\tau}
    =
    c_i^{\lambda\tau}\circ f_{ij}.$
    Denote $c^{\lambda\tau}:=\vpjl c_i^{\lambda\tau}
    \colon(\vpjl U_i)\cap(\vpjl V_i)\to\vpjl G_i.$
A \Emph{strict principal atlas} is a family $\mathcal T$ of \sis
principal trivializations which covers $\vpjl M_i$ and such that
$c^{\lambda\tau}$ is continuous for every
$\{\tau_i\},\{\lambda_i\}\in\mathcal T$. An inverse
system of principal bundles is \Emph{strict}
if it is equipped with a strict principal atlas.

\begin{proposition}\label{prop:strict-principal-bundles}
    Let
    $(\pi_i\colon P_i\to M_i,G_i,r_{ij},f_{ij},h_{ij})$
    be a strict inverse system of principal bundles with a strict
    principal atlas $\mathcal T$. Then
    $\mathcal T$ induces a strict inverse-system structure on
    $(P_i,r_{ij})$ for which $\{\pi_i\}$ is a split-compatible family
    of submersions. Moreover, 
            $\vpjl\pi_i\colon \vpjl P_i\to \vpjl M_i$ is a principal $\vpjl G_i$-bundle with
    action
    $(p_i)_i\cdot(g_i)_i:=(p_i\cdot g_i)_i$.
\end{proposition}

If $\{G_i\}$ is strong, the continuity condition in the definition of
a strict principal atlas is automatic. Indeed, every compatible family
$\{c_i^{\lambda\tau}\}$ induces a continuous map into the 
inverse-limit topology of $\vpjl G_i$, which agrees with its atlas topology.
If, in addition, $\{M_i\}$ is strong, then the induced system
$\{P_i\}$ is strong, since its limit chart domains are products of
open limit base and group chart domains.

\section{Strict inverse limits of Lie groupoids}\label{sec:algd-gpd}
We work in the smooth category, although we note that Theorem~\ref{thm:inv-lim-gpd-structure} applies in the $k\in \N$ case as well.

An inverse limit of Lie algebroids $(A_i\to M_i, \Phi_{ij},f_{ij})$
is \Emph{strict} (resp. \Emph{strong})
if it is strict (resp. {strong}) as an inverse system of smooth vector bundles.
\begin{theorem}\label{thm:invlimit-algebroid}
   Let  $(A_i\to M_i, \Phi_{ij},f_{ij})$ be a strict inverse system of Lie algebroids. Then $\vpjl A_i\to \vpjl M_i$ 
   carries a unique Lie algebroid structure for which 
   every limit map $(\Phi_i,f_i)\colon \vpjl A_i\to A_i$ is a Lie algebroid
   morphism. Moreover, if $B\to N$ is a Lie algebroid and  $(\Psi,g)\colon B\to \vpjl A_i$ is a smooth vector bundle map, then
   $(\Psi,g)$
   is a Lie algebroid morphism if and only if for every $i$, $\Phi_i\circ \Psi$ is a Lie algebroid morphism over $f_i\circ g$.
\end{theorem}
\begin{proof}
    Fix $(x_i)\in\vpjl M_i$ and let $\{\psi_i\colon A_i|_{U_i}\to U_i\times F_i\}$ be a \sis local trivialization overlapping 
    with a \sis chart about $(x_i)$. We suppress the chart and denote the modeling inverse system by $(M_{\psi_i},\epsij)$.
    Let $\anch_{\psi_i}\colon U_i\times F_i\to M_{\psi_i}$ and 
    $C_{\psi_i}\colon U_i\times F_i\times F_i\to F_i$ be the local representatives of the structure functions of $A_i$.
    Since $\Phi_{ij}$ is a Lie algebroid morphism and 
    it is locally given by $\nu_{ij}$ on the fiber, applying
    the dg condition to local linear functionals and locally constant one forms and  then using Hahn-Banach theorem, one gets 
    $\epsij\circ \anch_{\psi_j} = \anch_{\psi_i}\circ (f_{ij}\times \nu_{ij})$ and $\nu_{ij}\circ C_{\psi_j} = C_{\psi_i}\circ (f_{ij}\times \nu_{ij}\times \nu_{ij})$.
            Hence,
    $\vpjl \anch_{\psi_i}$ and
    $\vpjl C_{\psi_i}$  are
    smooth maps by Lemma~\ref{lemma:basic-C^k}.  
    Moreover, for all $y\in \vpjl U_i$, the map $(\vpjl C_{\psi_i})(y,\cdot,\cdot)$ is alternating bilinear since this 
    holds for $\nu_j\circ \vpjl C_{\psi_i}$, for every $j$.

    Set $U=\vpjl U_i$, $F=\vpjl F_i$, $\anch_\psi = \vpjl \anch_{\psi_i}$, and $C_\psi = \vpjl C_{\psi_i}$.
    For $\xi, \eta\in \Gamma((\vpjl A_i) |_U)$
    with representatives $\xi^\psi$ and $\eta^\psi$, set
    \begin{equation}\label{eq:Liebracketinv}
[\xi,\eta]^\psi(y) = d_1\eta^\psi(y)\anch_\psi(y,\xi^\psi(y)) - d_1\xi^\psi(y)\anch_\psi(y,\eta^\psi(y)) +  C_\psi(y, \xi^\psi(y), \eta^\psi(y)).
    \end{equation}
    This is smooth in $y$ and antisymmetric in $\xi$ and $\eta$. Let $\{\wt\psi_i\}$ be another \sis local trivialization about $(x_i)$. Write $\wt\psi_i\circ \psi_i^{-1} (y, v) = (y, T_i(y,v))$ where 
    the smooth map $T_i\colon (U_i\cap\wt U_i)\times F_i\to \wt F_i $ is linear in the second variable. The map $\vpjl T_i$ is smooth by Lemma~\ref{lemma:basic-C^k} and is linear in the second variable.
    Now for each $i$, the local structure maps $C_{\psi_i}$ and $C_{\wt\psi_i}$ are related by
    $$
    C_{\wt\psi_i}(z, T_i(z,v),T_i(z,w)) =  T_i(z, C_{\psi_i}(z,v,w))
    + d_1T_i(z,v)\anch_{\psi_i}(z,w) - d_1 T_i(z,w) \anch_{\psi_i}(z,v).
    $$
    Set $T=\vpjl T_i$.
    By taking the inverse limit and using the relations
    $\xi^{\wt\psi} (y) = T(y,\xi^\psi(y))$ (and similarly for $\eta$),
    we see that $[\xi,\eta]^{\wt\psi} (y) = T(y, [\xi,\eta]^\psi(y))$ 
    for all $y\in U$. That is, $[\xi,\eta]$ is a well-defined section of $A|_U$.
    Thus, $[\cdot,\cdot]_A\colon \shA\otimes_\R\shA\to \shA$ is a well-defined morphism of sheaves.
        Since every local trivialization is locally related to an \sis local trivialization by a smooth fiberwise linear transition map, evaluating the bracket on sections that are constant in the \sis trivialization gives the same change of trivialization formula for the local structure map. Hence the bracket has the form~\eqref{eq:Liebracketinv}, with a smooth structure map, in every local trivialization.

    For each $i$, the local structure maps $\anch_{\psi_i}$ and 
    $C_{\psi_i}$ satisfy
    \begin{equation}\label{eq:rho_id}
    d_1\anch_{\psi_i}(z,w)\anch_{\psi_i}(z,v) - d_1\anch_{\psi_i}(z,v)\anch_{\psi_i}(z,w) = \anch_{\psi_i}(z,C_{\psi_i}(z,v,w)),
    \end{equation}
    and
    \begin{equation}\label{eq:C_id}
    \sum_{\mathrm{cyc}(u,v,w)} \big(  
    d_1C_{\psi_i}(z,v,w)\anch_{\psi_i}(z,u)
+  C_{\psi_i}(z,u, C_{\psi_i}(z,v,w))
    \big) = 0,
    \end{equation}
    for $z\in U_i$ and $u,v,w\in F_i$.
          The same equations hold when we pass to the limit.
    Set $X(y) = \anch_{\psi}(y, \xi^\psi(y))$
    and $Y(y) = \anch_{\psi}(y, \eta^\psi(y))$. 
    By the chain rule and fiberwise linearity of $\anch_\psi$, one computes
    $dY(y)X(y) = d_1\anch_{\psi}(y,\eta^\psi(y))X(y) + \anch_\psi(y, 
    d\eta^\psi(y)X(y))$, and similarly for 
    $dX(y)Y(y)$. Now subtracting and using~\eqref{eq:rho_id}, we get
    $[X,Y](y) = \anch_\psi(y, [\xi,\eta]^\psi(y))$. That is,
    $\anch([\xi,\eta]) = [\anch(\xi),\anch(\eta)]$.
For $f\in C^\infty(U)$, using~\eqref{eq:Liebracketinv}, we see that 
$[\xi,f\eta] = f[\xi,\eta] + (\anch(\xi).f)\eta$. Let $\zeta\in \Gamma(A|_U)$. Expand $[\xi,[\eta,\zeta]]+[\eta,[\zeta,\xi]]+[\zeta,[\xi,\eta]]$ using~\eqref{eq:Liebracketinv}. After simplifying using symmetry of the second derivative, we get
\[
\sum_{\mathrm{cyc}(\xi,\eta,\zeta)} \big( 
d_1 C_\psi(y, \eta^\psi(y), \zeta^\psi(y))\anch_{\psi}(y,\xi^\psi(y)) + C_{\psi}(y,\xi^\psi(y), C_\psi(y,\eta^\psi(y),\zeta^\psi(y)))
\big) 
\]
which vanishes for all $y\in U$. Thus, $\vpjl A_i\to \vpjl M_i$ is a Lie algebroid.

Let $V\subseteq M_j$ be open and let $\omega\in \Omega^k_{A_j}(V)$.
Let $\{\psi_i\}$ be a \sis local trivialization such that $U:=\vpjl U_i\subseteq f_j^{-1}(V)$.
Now   $(\Phi_j^* \omega)_{\psi}(y; u_1,\ldots, u_k)
= \omega_{\psi_j}(f_j (y); \nu_j (u_1), \ldots, \nu_j (u_k))$.  Substituting   in~\eqref{eq:d_A}, then  from the definition of $\anch_\psi$ and $C_\psi$ it follows that $d_A(\Phi_j^*\omega) = \Phi^*(d_{A_j}\omega)$.

Now suppose that the vector bundle $\vpjl A_i\to \vpjl M_i$ admits another Lie algebroid structure for which each $\Phi_i$ is a Lie algebroid morphism. Then applying the dg map conditions to functions and locally constant one-forms gives uniqueness. Indeed,
locally in $\{\psi_i\}$, we have $\epsilon_j \circ \wt \anch_\psi =
\anch_{\psi_j}\circ (\epsilon_j\times \nu_j) =\epsilon_j \circ  \anch_\psi $ 
and $\nu_j\circ \wt C_{\psi} = C_{\psi_j}\circ (\epsilon_j\times \nu_j\times \nu_j)=\nu_j\circ C_{\psi}$,
for every $j$. Thus, 
$\wt\anch_\psi=\anch_\psi$ and
$\wt C_\psi = C_\psi$.

For the final statement, the forward direction is clear. Suppose that
$(\Phi_i\circ\Psi,f_i\circ g)$ is a Lie algebroid morphism for every
$i$. Let $\tau\colon B|_W\to W\times H$
be a local trivialization about $z\in N$, and let $\{\psi_i\}$ be a
\sis local trivialization about $g(z)$. Let
$\anch^B_\tau$ and $C^B_\tau$ denote the local structure maps of $B$.
Write $
    (\psi_i\circ(\Phi_i\circ\Psi)\circ\tau^{-1})(z,w)=(f_i(g(z)),K_i(z,w)),$
where $K_i$ is smooth and linear in the second variable. The maps
$K_i$ are compatible. Set $K=\vpjl K_i $. Then 
 $(\psi\circ\Psi\circ\tau^{-1})(z,w)=(g(z),K(z,w)).$

Since $(\Phi_i\circ\Psi,f_i\circ g)$ is a Lie algebroid morphism, 
   $ d_B(\Phi_i\circ\Psi)^*\omega_i
    =
    (\Phi_i\circ\Psi)^* d_{A_i}\omega_i$
for every local form $\omega_i$ on $A_i$. Applying this  to
functions and using \eqref{eq:d_A}, we get $
    \epsilon_i\bigl(
        dg(z)\anch^B_\tau(z,w)
    \bigr)
    =
    \anch_{\psi_i}
    \bigl(f_i(g(z)),K_i(z,w)\bigr).$
On the other hand, $
    \anch_{\psi_i}
(f_i(g(z)),K_i(z,w))
    =
    \epsilon_i(
        \anch_\psi(g(z),K(z,w))
    ).$
Hence
\begin{equation}\label{eq:limit-morphism-anchor}
    dg(z)\anch^B_\tau(z,w)
    =
    \anch_\psi(g(z),K(z,w)).
\end{equation}

Next, we consider the pullback of a locally constant one-form in the
trivialization $\psi_i$. Using \eqref{eq:d_A}, and the Hahn-Banach theorem, we obtain
\[
\begin{aligned}
    &C_{\psi_i}
    \bigl(f_i(g(z)),K_i(z,v),K_i(z,w)\bigr)\\
    &\qquad =
    K_i\bigl(z,C^B_\tau(z,v,w)\bigr)
    +d_1K_i(z,v)\anch^B_\tau(z,w)
    -d_1K_i(z,w)\anch^B_\tau(z,v).
\end{aligned}
\]
Now 
$C_{\psi_i}
    (f_i(g(z)),K_i(z,v),K_i(z,w)) =
    \nu_i (
    C_\psi(g(z),K(z,v),K(z,w))
    ).$
Hence, \begin{equation}\label{eq:limit-morphism-C}
\begin{aligned}
    C_\psi(g(z),K(z,v),K(z,w))
    &=
    K\bigl(z,C^B_\tau(z,v,w)\bigr)\\
    &\quad
    +d_1K(z,v)\anch^B_\tau(z,w)
    -d_1K(z,w)\anch^B_\tau(z,v).
\end{aligned}
\end{equation}

Let $\omega\in\Omega_{\vpjl A}^k(U)$, let $z\in W$, and let
$w_0,\ldots,w_k\in H$.  Then
\[
    (\Psi^*\omega)_\tau
    (z;w_1,\ldots,w_k)
    =
    \omega_\psi
    \bigl(g(z);K(z,w_1),\ldots,K(z,w_k)\bigr).
\]

Set $
    u_r:=K(z,w_r)$ and $
    X_r:=\anch^B_\tau(z,w_r).$
Then, by the chain rule,
\[
\begin{aligned}
&d_1(\Psi^*\omega)_\tau
 (z;w_0,\ldots,\widehat{w_a},\ldots,w_k)X_a\\
&\quad =
d_1\omega_\psi
 \bigl(g(z);u_0,\ldots,\widehat{u_a},\ldots,u_k\bigr)
 dg(z)X_a\\
&\qquad
+\sum_{r\neq a}
 \omega_\psi\bigl(
 g(z);
 u_0,\ldots,\widehat{u_a},\ldots,
 d_1K(z,w_r)X_a,\ldots,u_k
 \bigr).
\end{aligned}
\]
Using the \eqref{eq:d_A} for $d_B$ and
\eqref{eq:limit-morphism-anchor}, the $d_1\omega_\psi$ terms 
 become
\[
\sum_{a=0}^k(-1)^a
d_1\omega_\psi
 \bigl(g(z);u_0,\ldots,\widehat{u_a},\ldots,u_k\bigr)
 \anch_\psi(g(z),u_a).
\]
Since $\omega_\psi$ is alternating, the remaining terms containing
$d_1K$ and  the $C^B_\tau$ terms sum to
\[
\begin{aligned}
\sum_{0\leq a<b\leq k}(-1)^{a+b}
\omega_\psi\bigl(
g(z);
&K(z,C^B_\tau(z,w_a,w_b))
+d_1K(z,w_a)X_b-d_1K(z,w_b)X_a,\\
&u_0,\ldots,\widehat{u_a},\ldots,
\widehat{u_b},\ldots,u_k
\bigr).
\end{aligned}
\]
By~\eqref{eq:limit-morphism-C}, this is
\[
\sum_{0\leq a<b\leq k}(-1)^{a+b}
\omega_\psi\bigl(
g(z);
C_\psi(g(z),u_a,u_b),
u_0,\ldots,\widehat{u_a},\ldots,
\widehat{u_b},\ldots,u_k
\bigr).
\]
Hence
\[
\begin{aligned}
(d_B\Psi^*\omega)_\tau(z;w_0,\ldots,w_k)
&=(d_A\omega)_\psi(g(z);u_0,\ldots,u_k)=(\Psi^* d_A\omega)_\tau(z;w_0,\ldots,w_k).
\end{aligned}
\]
Hence $d_B\circ\Psi^\sharp=\Psi^\sharp\circ g^{-1} d_A$. Thus,
$(\Psi,g)$ is a Lie algebroid morphism.
\end{proof}

For a discussion of inverse systems of Lie algebroids in the
convenient settings; see~\cite{CP23}.

\begin{definition} An inverse system of Lie groupoids $(\G_i\rightrightarrows M_i, \F_{ij}, f_{ij} )$ is \Emph{strict} if:
\begin{enumerate}
    \item $(\G_i, \F_{ij})$ and $(M_i, f_{ij})$ are strict inverse
    systems of manifolds;
    \item $\{\src_i\}$ is a split-compatible family of submersions;
    \item $\{\mul_i\}$ and $\{\inv_i\}$ are morphisms of 
    strict 
    inverse systems, where $\{\mset{\G_i}\}$ carries the split strict inverse system of submanifolds structure of Proposition~\ref{prop:submersions}.
\end{enumerate}
The system is \Emph{strong} if $\{\G_i\}$ is strong. 
\end{definition}
The second condition can equivalently be  written in terms of  $\{\tgt_i\}$.
The third condition is redundant if the system is strong
 since $\vpjl \mul_i$ and $\vpjl \inv_i$ are automatically 
continuous. Finally, we note that $\{M_i\}$ is strong if $\{\G_i\}$ is strong. Indeed, one can start with split-compatible submersion charts for $\{\src_i\}$ and show that $\{\one_i\}$ is a split-compatible family of embeddings then use Proposition~\ref{prop:immersions}.

\begin{theorem}\label{thm:inv-lim-gpd-structure}
    Let $(\G_i\rightrightarrows M_i, \F_{ij}, f_{ij} )$
    be  a strict inverse system of Lie groupoids. Then
    $\vpjl \G_i \rightrightarrows \vpjl M_i$ carries a unique
    Lie groupoid structure defined by the maps 
    $\vpjl \src_i$, $\vpjl \tgt_i$, $\vpjl \mul_i$, $\vpjl \one_i$, and $\vpjl \inv_i$.
\end{theorem}
\begin{proof}
    By Proposition~\ref{prop:submersions}, $\vpjl \src_i$
    is a submersion. Since $\vpjl \inv_i$ is smooth and $(\vpjl \inv_i)^2 = \id$, it is a diffeomorphism. Hence $\vpjl \tgt_i=\vpjl \src_i\circ\vpjl \inv_i$ is a submersion.
    The map $\vpjl \mul_i$ is smooth and defines smooth multiplication since $\mset{(\vpjl \G_i)}\cong \vpjl (\mset{\G_i})  $ by Proposition~\ref{prop:submersions}.
    Fix $(x_i)\in\vpjl M_i$.     Take a local section $\sigma\colon U\to \vpjl \G_i$
    of $\vpjl \src_i$ through $(1_{x_i})$. Then $(\vpjl \one_i)|_{U} = (\vpjl \mul_i)\circ ( (\vpjl \inv_i)\circ \sigma, \sigma) $
    is smooth. This shows that $\vpjl \one_i$ is smooth.
    Since $\vpjl \one_i$ is a section of $\vpjl \src_i$, it is a smooth embedding.
Composition with the limit maps shows that the
groupoid axioms hold and that the groupoid structure is unique.
    \end{proof}

\begin{theorem}\label{thm:inv-lie-groupoid}
         Let $(\G_i\rightrightarrows M_i, \F_{ij}, f_{ij} )$
    be  a strict inverse system of Lie groupoids.
    Then $(\Lie(\G_i)\to M_i, \Lie(\F_{ij}), f_{ij})$
    is a strict inverse system of Lie algebroids. Moreover, we have a
    canonical Lie algebroid isomorphism $\Lie(\vpjl \G_i)\cong \vpjl \Lie(\G_i)$.
\end{theorem}
\begin{proof}
    Lemmas~\ref{lemma:kernel-subbundle} and
    \ref{lemma:pullback-sum-bundles}, and 
    \cite[Corollary~3.7]{Sha26}, show that
    $\{\Lie(\G_i)\}$ is a strict inverse system of Lie algebroids and give
    a canonical vector bundle isomorphism
    $\Xi\colon \Lie(\vpjl \G_i)\to \vpjl \Lie(\G_i).$
        Since
    $\Phi_j\circ\Xi=\Lie(\F_j)$, and each $\Lie(\F_j)$ is a Lie
    algebroid morphism, the final part of
    Theorem~\ref{thm:invlimit-algebroid} shows that $\Xi$ is a Lie
    algebroid morphism. Since $\Xi$ is a vector bundle isomorphism, it is
    a Lie algebroid isomorphism.
\end{proof}

\paragraph{Functoriality}
Denote by
$\SInvMan_I^k$ the category of strict $C^k$-inverse systems of manifolds, by
 $\SInvVBun^k_I$ the category of strict $C^k$-inverse systems of vector bundles, by $\SInvLieGpd_I$ the category of strict inverse systems of Lie groupoids, and by $\SInvLieAlg_I$ the category of strict inverse systems of Lie algebroids. The results in Section~\ref{sec:mfds_bundles} show that
inverse limits
define a functor
$\vpjl\colon \SInvMan_I^k\to\Man^k$.
Finite products exist in $\SInvMan_I^k$ and are preserved by $\vpjl$.
Moreover, if two morphisms of strict inverse systems have a common target and
one of them is a split-compatible family of submersions, then their levelwise
pullbacks form a strict inverse system and the canonical map
$\vpjl(M_i\times_{N_i}\wt M_i)\to
(\vpjl M_i)\times_{\vpjl N_i}(\vpjl\wt M_i)$
is a $C^k$-diffeomorphism. We will use this restricted pullback statement for
the manifolds of composable arrows of Lie groupoids.

The analogous constructions for vector bundles are also compatible with
inverse limits. In particular, pullback of a strict inverse system of vector
bundles along a morphism of strict inverse systems commutes canonically with
$\vpjl$, as do fiberwise direct sums. The tangent construction
is functorial on strict inverse systems and gives a functor
$T\colon\SInvMan_I^k\to\SInvVBun_I^{k-1}$. The canonical identification
$T(\vpjl M_i)\cong\vpjl TM_i$ is natural. 
For $k\geq2$, the tangent functor preserves the restricted pullbacks of strict inverse systems of manifolds described above.

Applying the Lie functor levelwise defines a functor
$\Lie_I\colon\SInvLieGpd_I\to\SInvLieAlg_I$. For a strict inverse system
$\{\G_i\rightrightarrows M_i\}$, let
$\Xi_{\G}\colon \Lie(\vpjl \G_i)\to\vpjl \Lie(\G_i)$ denote the canonical
isomorphism constructed above. If
$\{\mathcal F_i\}\colon\{\G_i\}\to\{\mathcal H_i\}$ is a morphism of strict inverse
systems of Lie groupoids, then after composing with limit maps, we can show that
$\Xi_{\mathcal H}\circ \Lie(\vpjl \mathcal F_i)
=(\vpjl \Lie(\mathcal F_i))\circ\Xi_{\mathcal G}$.
Hence the family $\{\Xi_{\G}\}$ is natural, and there is a natural
isomorphism
$\Lie\circ\vpjl\cong\vpjl\circ \Lie_I$.

The integration results below should not be interpreted as providing an
inverse functor to $\Lie_I$. Starting from an inverse system of integrable
Banach-Lie algebroids, source-simply connected integrations and the
connecting groupoid morphisms are determined levelwise by Lie II. 
However, Lemma~\ref{lemma:levelwise-integrations} only produces an inverse system of Banach-Lie groupoids. Indeed, additional geometric hypotheses are required in Theorem~\ref{thm:integration-inverse-system} to ensure that this inverse system is strict.
Even after
a strict inverse-limit integration has been obtained, source-simply
connectedness of the limit is a separate homotopy issue governed
for the inverse sequences considered below by the derived inverse limit
$\vpjl^1\pi_2$ of the source fibers.

\section{Integration of inverse systems of integrable Banach-Lie algebroids}\label{sec:integration}

We first construct the source-simply connected integrations of an inverse
system of integrable Banach-Lie algebroids. The subsequent atlas
criteria apply to any compatible system of Banach-Lie groupoid
integrations. Source-simply connectedness of the limit is a separate
question. Throughout,
base manifolds and source fibers are Hausdorff; arrow manifolds need not
be Hausdorff. An integrable Banach-Lie algebroid
is one admitting a Banach-Lie groupoid integration with Hausdorff source
fibers. Simply connected means path connected with trivial fundamental
group.

\begin{lemma}\label{lemma:levelwise-integrations}
Let $(A_i\to M_i,\Phi_{ij},f_{ij})$ be an inverse system of integrable
Banach-Lie algebroids. Then there is an inverse system
$(\G_i\rightrightarrows M_i,\F_{ij},f_{ij})$ of source-simply connected
Banach-Lie groupoids such that $\Lie(\G_i)\cong A_i$ and, under these
identifications, $\Lie(\F_{ij})=\Phi_{ij}$.
\end{lemma}
\begin{proof}
By~\cite[Theorem~5.1]{BGJP19}, each $A_i$ admits a source-simply
connected integration $\G_i$. Indeed, starting from an integration with
Hausdorff source fibers, the construction takes universal coverings of
the source fibers of its source-connected subgroupoid, so these fibers
remain Hausdorff. Lie II~\cite[Theorem~4.8]{Sha26} integrates $\Phi_{ij}$
to a unique morphism $\F_{ij}\colon\G_j\to\G_i$ over $f_{ij}$.
Uniqueness gives $\F_{ii}=\id$ and $\F_{ik}=\F_{ij}\circ\F_{jk}$.
\end{proof}

The next two lemmas construct compatible source charts from compatible
algebroid trivializations. Source-simply connectedness is not needed for
this construction.

\begin{lemma}\label{lemma:normalized-source-chart}
Let $\G\rightrightarrows M$ be a Banach-Lie groupoid and
$\psi\colon\Lie(\G)|_{U_\psi}\to U_\psi\times F$ a local
trivialization about $x\in M$. There are open neighborhoods
$U\subseteq U_\psi$ of $x$ and $W\subseteq F$ of $0$, and a
diffeomorphism $\Theta\colon U\times W\to\Omega$ onto an open
neighborhood of $\one(U)$ such that $\Theta(y,0)=1_y$,
$\src\circ\Theta=\pr_1$,
$d_2\Theta(y,0)v=\psi_y^{-1}(v)$, and
$\tgt(\Omega)\subseteq U_\psi$.
\end{lemma}
\begin{proof}
For $v\in F$, put $\xi_v(y):=\psi_y^{-1}(v)$ and let $X_v$ be its
right-invariant source-vertical extension on $\tgt^{-1}(U_\psi)$.
Smooth dependence of Banach flows gives
$\Theta(y,v):=\Fl_1^{X_v}(1_y)$ near $(x,0)$.
Since $X_0=0$ and $X_{tv}=tX_v$, we have
$\Theta(y,0)=1_y$, $\src\circ\Theta=\pr_1$, and
$d_2\Theta(y,0)v=\psi_y^{-1}(v)$. Thus
$T_{(x,0)}\Theta(u,v)=T_x\one(u)+\psi_x^{-1}(v)$,
a Banach-space isomorphism. The inverse function theorem, followed
by shrinking $U$ and $W$, proves the assertion.
\end{proof}

\begin{lemma}\label{lemma:compatible-source-charts}
Let $(\G_i\rightrightarrows M_i,\F_{ij},f_{ij})$ be an inverse
system of Banach-Lie groupoids, put $A_i:=\Lie(\G_i)$ and
$\Phi_{ij}:=\Lie(\F_{ij})$, and fix $(x_i)\in\vpjl M_i$.
Suppose there are compatible local trivializations $\psi_i\colon A_i|_{U_i'}\to U_i'\times F_i$ about
$x_i$, with fiber connecting maps $\nu_{ij}$. Suppose that the
normalized source charts of Lemma~\ref{lemma:normalized-source-chart}
are defined on $U_i\times W_i$, where $U_i\subseteq U_i'$,
$f_{ij}(U_j)\subseteq U_i$, $\nu_{ij}(W_j)\subseteq W_i$, and
$\vpjl W_i$ is open in $\vpjl F_i$. Then their flow constructions
satisfy $\F_{ij}\circ\Theta_j=\Theta_i\circ(f_{ij}\times\nu_{ij})$.
If there are compatible base charts
$\phi_i\colon U_i\to M_{\phi_i}$ with
$\vpjl\phi_i(U_i)$ open, then
$\chi_i:=(\phi_i\times\id_{F_i})\circ\Theta_i^{-1}$
form a \sis chart about $(1_{x_i})$, in which
$\phi_i\circ\src_i\circ\chi_i^{-1}(u,v)=u$.
\end{lemma}
\begin{proof}
For $v\in W_j$, compatibility gives
$\Phi_{ij}\circ\xi_v^j=\xi_{\nu_{ij}(v)}^i\circ f_{ij}$.
The corresponding right-invariant fields are $\F_{ij}$-related.
Uniqueness of flows gives the asserted identity at time $1$.
Consequently $\chi_i\circ\F_{ij}=(\epsij\times\nu_{ij})\circ\chi_j$,
and their inverse-limit image is
$(\vpjl\phi_i(U_i))\times\vpjl W_i$, which is open.
The source formula follows from $\src_i\circ\Theta_i=\pr_1$.
\end{proof}

We use compatible local source sections to translate the unit charts
to arbitrary arrows. These sections need not be bisections.

\begin{theorem}\label{thm:integration-inverse-system}
Let $(\G_i\rightrightarrows M_i,\F_{ij},f_{ij})$ be an inverse
system of Banach-Lie groupoids, and put $A_i:=\Lie(\G_i)$ and
$\Phi_{ij}:=\Lie(\F_{ij})$. Suppose every
$x\in\vpjl M_i$ admits data as in
Lemma~\ref{lemma:compatible-source-charts}, and the resulting base
\sis charts form a strict atlas $\mathcal A_M$.
For every $(g_i)\in\vpjl\G_i$, put $x_i:=\src_i(g_i)$ and
$y_i:=\tgt_i(g_i)$. Suppose there are a base chart family
$\{\phi_i\colon V_i\to M_{\phi_i}\}\in\mathcal A_M$ about $x$,
compatible normalized source charts
$\Theta_i\colon U_i\times W_i\to\Omega_i$ about $1_{y_i}$, and
compatible local source sections $\sigma_i\colon V_i\to\G_i$
satisfying $\src_i\circ\sigma_i=\id_{V_i}$,
$\sigma_i(x_i)=g_i$, $\tgt_i(\sigma_i(V_i))\subseteq U_i$, and
$\F_{ij}\circ\sigma_j=\sigma_i\circ f_{ij}$.
Put $\Theta_i^{g_i}(z,v):=
\Theta_i(\tgt_i(\sigma_i(z)),v)\sigma_i(z)$ and
$\Omega_i^{g_i}:=\Theta_i^{g_i}(V_i\times W_i)$.
Assume that the inverse-limit overlap images of these chart families
are open and that, for their atlas topology,
$\vpjl\inv_i$ and $\vpjl\mul_i$ are continuous.

Then $(\G_i\rightrightarrows M_i,\F_{ij},f_{ij})$ is a strict inverse
system of Lie groupoids, $(A_i\to M_i,\Phi_{ij},f_{ij})$ is a strict
inverse system of Lie algebroids, and
$\Lie(\vpjl\G_i)\cong\vpjl A_i$ canonically. In particular,
$\vpjl\G_i\rightrightarrows\vpjl M_i$ integrates $\vpjl A_i$.
\end{theorem}
\begin{proof}
The compatible algebroid trivializations give a strict vector-bundle
atlas. Define $\rho_i(k):=k\sigma_i(\src_i(k))^{-1}$ on
$\src_i^{-1}(V_i)$. Then
$\Omega_i^{g_i}=\rho_i^{-1}(\Omega_i)$ is open and
$\Theta_i^{g_i}\colon V_i\times W_i\to\Omega_i^{g_i}$ is a
diffeomorphism, with inverse
$k\mapsto(\src_i(k),\pr_2\Theta_i^{-1}(\rho_i(k)))$.
Compatibility gives
$\F_{ij}\circ\Theta_j^{g_j}
=\Theta_i^{g_i}\circ(f_{ij}\times\nu_{ij})$.
Thus $\chi_i^{g_i}:=(\phi_i\times\id_{F_i})
\circ(\Theta_i^{g_i})^{-1}$ are compatible charts with open
inverse-limit image $(\vpjl\phi_i(V_i))\times\vpjl W_i$.
The overlap hypothesis makes them a strict atlas. The source maps
are split-compatible, since their local form is $(u,v)\mapsto u$.
The continuity assumptions make inversion and multiplication
morphisms of strict inverse systems. Theorems~\ref{thm:inv-lie-groupoid}
and~\ref{thm:invlimit-algebroid} now give the assertion.
\end{proof}

The strictness and continuity hypotheses of
Theorem~\ref{thm:integration-inverse-system} are automatic when the
translated chart domains are open in the ordinary inverse-limit
topology. In this case a strict base atlas need not be prescribed in
advance.

\begin{corollary}\label{cor:strong-integration}
Let $(\G_i\rightrightarrows M_i,\F_{ij},f_{ij})$
be an inverse system of Banach-Lie groupoids,
and put $A_i:=\Lie(\G_i)$ and
$\Phi_{ij}:=\Lie(\F_{ij})$.
Suppose every $x\in\vpjl M_i$ admits the data of
Lemma~\ref{lemma:compatible-source-charts}, including compatible base
charts $\phi_i$ and normalized source charts $\Theta_i$. Suppose
further that every $(g_i)\in\vpjl\G_i$ admits compatible source
sections and translated charts $\Theta_i^{g_i}$ as in
Theorem~\ref{thm:integration-inverse-system}, using these base charts,
with $\vpjl\Omega_i^{g_i}$ open in the ordinary inverse-limit
topology. At the units choose $\sigma_i=\one_i|_{U_i}$ and the
prescribed normalized charts. Then the levelwise integrations form a
strong inverse system of Lie groupoids, and
$\Lie(\vpjl\G_i)\cong\vpjl A_i$ canonically. If $I$ has a countable
cofinal subset, the limit is a Fr\'echet-Lie groupoid integrating the
Fr\'echet-Lie algebroid $\vpjl A_i$.
\end{corollary}

For an inverse sequence, a simple sufficient condition is eventual
stabilization: for each $g=(g_n)\in\vpjl\G_n$, suppose some $n_0$
satisfies $\Omega_m^{g_m}=\F_{n_0m}^{-1}(\Omega_{n_0}^{g_{n_0}})$
for all $m\geq n_0$. Compatibility then gives
$\vpjl\Omega_n^{g_n}=\pr_{n_0}^{-1}(\Omega_{n_0}^{g_{n_0}})$,
which is open in the inverse limit topology. Thus the preceding corollary recovers
the stabilization criterion without separate overlap or continuity
assumptions.

We next examine source-simply connectedness. Unlike strictness,
it depends on the topology of the source fibers and does not follow
from the levelwise integrations.

\begin{proposition}\label{prop:source-fiber-homotopy}
Let $(\G_n\rightrightarrows M_n,\F_{nm},f_{nm})$ be an inverse
sequence whose limit is a Lie groupoid with componentwise structural
maps. Fix $x=(x_n)\in\vpjl M_n$ and put $X_n:=\G_{n,x_n}$,
based at $1_{x_n}$. Suppose the canonical bijection
$(\vpjl\G_n)_x\cong\vpjl X_n$ is a homeomorphism for the
inverse limit topology on the right, and that the restricted
connecting maps $X_{n+1}\to X_n$ are locally trivial bundles.
Then there is a natural exact sequence
\[
1\longrightarrow\vpjl\nolimits^1\pi_2(X_n,1_{x_n})
\longrightarrow\pi_1((\vpjl\G_n)_x,1_x)
\longrightarrow\vpjl\pi_1(X_n,1_{x_n})
\longrightarrow1.
\]
If each $X_n$ is simply connected, then $(\vpjl\G_n)_x$ is
path connected. The homeomorphism hypothesis is automatic if the arrow
system is strong.
\end{proposition}
\begin{proof}

The restricted connecting maps are locally
trivial bundles and preserve the base points.
Thus \cite[Theorem~2.1]{Hir15}, applied in
degree one, gives
\[
1\longrightarrow\vpjl\nolimits^1\pi_2(X_n)
\longrightarrow\pi_1(\vpjl X_n)
\longrightarrow\vpjl\pi_1(X_n)
\longrightarrow1.
\]
The asserted sequence follows from the
homeomorphism hypothesis. If every $X_n$
is simply connected, the degree-zero part
of the same theorem shows that $\vpjl X_n$
is path connected. Finally, if the arrow
system is strong, both topologies on the
limit source fiber coincide with the
subspace topology inherited from
$\vpjl\G_n$.
\end{proof}

\begin{corollary}\label{cor:ssc-limit}
Under the hypotheses of Theorem~\ref{thm:integration-inverse-system},
suppose each $\G_n$ is source-simply connected and the source
fibers satisfy the additional hypotheses of
Proposition~\ref{prop:source-fiber-homotopy}. Then every limit source
fiber is path connected and
$\pi_1((\vpjl\G_n)_x,1_x)
\cong\vpjl\nolimits^1\pi_2(\G_{n,x_n},1_{x_n})$
canonically. Hence $\vpjl\G_n$ is source-simply connected if and
only if $\vpjl^1\pi_2(\G_{n,x_n},1_{x_n})=0$ for every
$(x_n)\in\vpjl M_n$.
\end{corollary}

\begin{proof}

Since every $\G_n$ is source-simply
connected, Proposition~\ref{prop:source-fiber-homotopy}
gives both path connectedness and the
asserted fundamental group.
\end{proof}

\begin{corollary}\label{cor:source-fiber-ML}
Under the hypotheses of Corollary~\ref{cor:ssc-limit}, suppose that
$\{\pi_2(\G_{n,x_n},1_{x_n})\}_n$ satisfies the
Mittag-Leffler condition for every $(x_n)\in\vpjl M_n$.
Then $\vpjl\G_n$ is source-simply connected. It suffices that
the restricted connecting maps induce surjections on $\pi_2$; in
particular, this holds if the fibers of the restricted connecting maps
are simply connected.
\end{corollary}
\begin{proof}
The groups $\pi_2(X_n,1_{x_n})$ are abelian, therefore the
Mittag-Leffler condition gives $\vpjl^1\pi_2(X_n)=0$ by
\cite[Corollary~IX.3.5]{BK72}. The last assertion follows from
the homotopy exact sequence of a locally trivial bundle.
\end{proof}

In particular, under the hypotheses of
Corollary~\ref{cor:strong-integration}, if each $\G_n$ is
source-simply connected, the restricted
source-fiber connecting maps are locally
trivial bundles, and the inverse sequences
of second homotopy groups satisfy the
Mittag--Leffler condition, then
$\vpjl A_n$ admits a source-simply
connected Fr\'echet-Lie groupoid
integration.

\section{Examples}\label{sec:examples}
The following examples illustrate strict and strong inverse limits and
applications of the integration results. Example~\ref{ex:jetgpd} treats
jet groupoids, while Example~\ref{ex:dyadic} shows that strictness
need not imply strongness. Example~\ref{ex:current-groupoids}
concerns current groupoids. Example~\ref{ex:gauge-groupoids}
constructs inverse-limit gauge groupoids integrating Atiyah algebroids
and applies the source-fiber homotopy criterion. Finally,
Example~\ref{ex:hopf-gauge} gives a source-simply connected
Fr\'echet integration whose Atiyah anchor admits no Lie algebroid
splitting.

\begin{example}[Infinite jet groupoids]\label{ex:jetgpd}
    Let $\G\rightrightarrows M$ be a finite-dimensional 
    Lie groupoid,
    with $\dim M=n$ and $\dim\G-n=d$. Put $J^0\G:=\G$, and for
    $k\geq1$ let $J^k\G$ be the Lie groupoid of $k$-jets of local
    bisections of $\G$; see, for example,
    \cite{Mei17}. The truncations
    $\pi_l^k(j_x^k\sigma):=j_x^l\sigma$, $l\leq k$, are Lie groupoid
    morphisms.

    Fix $(z_k)_k\in\vpjl J^k\G$. For each $k$, write
$z_k=j_x^k\sigma_k$ for a local bisection $\sigma_k$. Put
$g:=z_0=\sigma_k(x)$. Choose a chart $\phi$ about $x$ and a
source-adapted chart $\theta=(\phi\circ\src,\theta_2)$ about $g$.
A local section $\tau$ of $\src$ is represented by
$\theta\circ\tau\circ\phi^{-1}(y)=(y,h_\tau(y))$.
The induced $k$-jet chart records the base point, the value of
$h_\tau$, and its derivatives up to order $k$.

    Put $E_r:=\R^d\otimes S^r(\R^n)^*$. For $k\geq1$, the bisection
    condition is open and depends only on the first jet, so the charts
    may be chosen with image
    $\mathcal O\times\prod_{r=2}^kE_r$, where
    $\mathcal O\subseteq\R^n\times\R^d\times E_1$ is open.
    Truncation is coordinate projection, and the inverse-limit chart
    image is $\mathcal O\times\prod_{r=2}^\infty E_r$, which is open in
    the product topology. The same description applies to chart
    overlaps. Hence the compatible jet charts form a strong atlas.
    In these charts the source maps are coordinate projections, so
    they are split-compatible. Therefore
    $\{J^k\G\rightrightarrows M\}_k$ is a strong inverse sequence of
    Lie groupoids.

    Put $J^\infty\G:=\vpjl_kJ^k\G$. Then $J^\infty\G$ is a
    Fr\'echet-Lie groupoid. The standard identifications
    $\Lie(J^k\G)\cong J^k\Lie(\G)$ are compatible with truncation, and
    Theorem~\ref{thm:inv-lie-groupoid} gives
    $\Lie(J^\infty\G)\cong J^\infty\Lie(\G)$.
\end{example}

\begin{example}[Action groupoids]
Let $(M_i,g_{ij})$ be a strict inverse system of manifolds and
$(G_i,h_{ij})$ a strict inverse system of Lie groups. Suppose that
smooth right actions $\alpha_i\colon M_i\times G_i\to M_i$
form a morphism of strict inverse systems of manifolds. Then the
action groupoids $M_i\rtimes G_i\rightrightarrows M_i$ form a strict
inverse system of Lie groupoids with connecting morphisms
$g_{ij}\times h_{ij}$. Indeed, the source maps are projections
and hence split-compatible, while multiplication and inversion
are morphisms of strict inverse systems. Consequently,
$\vpjl(M_i\rtimes G_i)\cong(\vpjl M_i)\rtimes(\vpjl G_i)$
canonically as Lie groupoids, with limit action $\vpjl\alpha_i$.
\end{example}

\begin{example}[A nonstrong action groupoid]\label{ex:dyadic}
Let $(H_n,p_{nm})$ be the inverse sequence $H_n=S^1$ with
$p_{nm}(a)=a^{2^{m-n}}$ whenever $n\leq m$, and equip
$\vpjl H_n$ with the strict, nonstrong Lie group atlas
of Section~\ref{sec:mfds_bundles}. Put $M_n:=S^1$, with
identity connecting maps, and define smooth right actions
$\alpha_n\colon M_n\times H_n\to M_n$ by
$\alpha_n(z,a)=za^{2^n}$. These actions are compatible,
since $\alpha_n(z,p_{nm}(a))=za^{2^m}=\alpha_m(z,a)$.
Thus the action groupoids
$\G_n:=M_n\rtimes H_n\rightrightarrows M_n$
form a strict inverse sequence, and
$\vpjl\G_n\cong S^1\rtimes(\vpjl H_n)\rightrightarrows S^1$,
with limit action $\alpha(z,(a_n))=za_1^2$.
This action is transitive, since for any $w\in S^1$ one can
choose $t\in\R$ with $e^{it}=w$ and put $a_n=e^{it/2^n}$.
The arrow atlas is not strong: its topology restricts to the
strict, nonstrong topology of $\vpjl H_n$ on every source
fiber. These source fibers are diffeomorphic to $\vpjl H_n$
and hence disconnected. Thus strictness implies neither
strongness nor source-connectedness, even over a connected
finite-dimensional base.
\end{example}

\begin{example}[Current groupoids]\label{ex:current-groupoids}
    Let $K$ be a compact smooth manifold and let $l\in\N$.
    Let $(\G_i\rightrightarrows M_i,\F_{ij},f_{ij})$ be a strong
    inverse system of Lie groupoids with Hausdorff arrow manifolds
    and Banach base manifolds. Put $\G:=\vpjl\G_i$ and $M:=\vpjl M_i$.
    Suppose that the arrow and base systems admit compatible local
    additions $\Sigma_i^\G$ and $\Sigma_i^M$ whose inverse limits are
    local additions on $\G$ and $M$, respectively. Assume moreover that
    these local additions are adapted to the source maps in the sense
    that $\src_i\circ\Sigma_i^\G=\Sigma_i^M\circ T\src_i$ wherever
    both sides are defined.

    We briefly indicate why the mapping spaces again form strong
    inverse systems. Let $\{N_i\}$ denote either the arrow or the base
    system and put $N:=\vpjl N_i$. For a compatible family
    $(h_i)\in\vpjl C^l(K,N_i)$, the map $h\colon K\to N$,
    $h(x):=(h_i(x))_i$, is $C^l$. The standard local-addition charts
    about the $h_i$ are compatible with the connecting maps. Their
    model spaces are $\Gamma_{C^l}(h_i^*TN_i)$, and a finite compatible
    trivializing cover of $K$ identifies
    $\Gamma_{C^l}(h^*TN)$ with
    $\vpjl\Gamma_{C^l}(h_i^*TN_i)$ as locally convex spaces.
    Moreover, compactness of $K$ and strongness of $\{N_i\}$ imply,
    by the usual finite-cover argument, that the inverse-limit chart
    domains are open in the inverse-limit topology. Thus the canonical
    map $C^l(K,N)\to\vpjl C^l(K,N_i)$,
    $h\mapsto(\pr_i\circ h)_i$, is a smooth diffeomorphism, and the
    mapping-space chart families form a strong atlas.

    We also record the split-compatibility argument that will be used
    for the source maps.

    \medskip
    \noindent\Emph{Lemma.}
    Let $(N_i,a_{ij})$ and $(L_i,b_{ij})$ be strong inverse systems of
    manifolds and let $\{q_i\colon N_i\to L_i\}$ be a split-compatible
    family of submersions. Suppose that the two systems admit
    compatible local additions $\Sigma_i^N$ and $\Sigma_i^L$ adapted
    to the $q_i$. Then the family
    $\{C^l(K,q_i)\colon C^l(K,N_i)\to C^l(K,L_i)\}$ is
    split-compatible with respect to the induced mapping-space charts.

    \medskip
    \noindent\emph{Proof.}
    Fix a compatible family $(h_i)\in\vpjl C^l(K,N_i)$ and put
    $k_i:=q_i\circ h_i$. Let $\kappa_i^N$ and $\kappa_i^L$ be the
    local-addition charts about $h_i$ and $k_i$, respectively. Since
    the local additions are adapted to $q_i$, these charts represent
    $C^l(K,q_i)$ by the continuous linear map
    $\Lambda_i\colon\Gamma_{C^l}(h_i^*TN_i)\to
    \Gamma_{C^l}(k_i^*TL_i)$ given by
    $\Lambda_i(X)(x):=T_{h_i(x)}q_i(X(x))$.

    For every $x\in K$, split-compatibility of $\{q_i\}$ gives
    compatible charts about $(h_i(x))_i$ in which $q_i$ is the
    projection $(u,c)\mapsto u$. After restricting to a neighborhood
    $W_x$ of $x$, these charts therefore induce compatible smooth
    bundle maps
    $\rho_i^x\colon k_i^*TL_i|_{W_x}\to h_i^*TN_i|_{W_x}$
    satisfying $Tq_i\circ\rho_i^x=\id$. Choose a finite cover
    $\{W_a\}_{a=1}^m$ of this form and  denote the
    corresponding bundle maps by $\rho_i^a$. Choose a smooth partition of
    unity $\{\chi_a\}_{a=1}^m$ subordinate to it. The bundle maps
    $\rho_i:=\sum_a\chi_a\rho_i^a$, with each summand extended by zero
    outside $W_a$, are compatible and satisfy
    $Tq_i\circ\rho_i=\id$.

    Hence the maps
    $R_i\colon\Gamma_{C^l}(k_i^*TL_i)\to
    \Gamma_{C^l}(h_i^*TN_i)$, $R_i(Y):=\rho_i\circ Y$, form a
    compatible family of continuous linear right inverses of
    $\Lambda_i$. Therefore
    $P_i:=\id-R_i\Lambda_i$ is a compatible family of continuous
    projections onto $\ker\Lambda_i$. In particular,
    $\{\ker\Lambda_i\}$ is a complemented inverse subsystem.
    The compatible linear isomorphisms
    $X\mapsto(\Lambda_iX,P_iX)$ identify
    $\Gamma_{C^l}(h_i^*TN_i)$ with
    $\Gamma_{C^l}(k_i^*TL_i)\times\ker\Lambda_i$, with inverse
    $(Y,Z)\mapsto R_iY+Z$. Composing the source mapping-space charts
    with these isomorphisms represents $C^l(K,q_i)$ by the projection
    onto the first factor. Thus $\{C^l(K,q_i)\}$ is split-compatible.
    \hfill$\square$

    \medskip

    Applying the lemma to $q_i=\src_i$ shows that the maps
    $C^l(K,\src_i)\colon C^l(K,\G_i)\to C^l(K,M_i)$ form a
    split-compatible family. By~\cite[Theorem~3.2]{AGS20},
    $C^l(K,\G_i)\rightrightarrows C^l(K,M_i)$ is a Lie groupoid with
    pointwise structural maps. Hence the current groupoids form a
    strong inverse system whose inverse limit is canonically
    $C^l(K,\G)\rightrightarrows C^l(K,M)$. Moreover,
    \cite[Theorem~4.6]{AGS20} gives
    $\Lie(C^l(K,\G_i))\cong C^l(K,\Lie(\G_i))$. Therefore
    Theorem~\ref{thm:inv-lie-groupoid} gives
    \[
        \Lie(C^l(K,\G))
        \cong
        \vpjl C^l(K,\Lie(\G_i)).
    \]

    In particular, let $\G\rightrightarrows M$ be finite dimensional.
We can choose local additions $\Sigma^\G$ and $\Sigma^M$ adapted to
$\src$, so that $\src\circ\Sigma^\G=\Sigma^M\circ T\src$ wherever defined; see \cite[Proposition 4.2]{KS25}.
    Consider the regularity inverse sequence
    $\{C^r(K,\G)\rightrightarrows C^r(K,M)\}_{r\in\N}$, with the
    canonical inclusions as connecting maps. At a smooth map, the
    standard local-addition chart domains for $C^s$ are inverse images
    of the corresponding $C^r$ chart domains whenever $r\leq s$.
    Hence the arrow and base mapping spaces form strong inverse
    sequences. 
    These local additions induce adapted local-addition charts at every regularity, so the preceding lemma shows that the current source maps form a split-compatible family.
     Therefore the current
    groupoids form a strong inverse sequence with inverse limit
    $C^\infty(K,\G)\rightrightarrows C^\infty(K,M)$, and
    \[
        \Lie(C^\infty(K,\G))
        \cong
        \vpjl_{r\in\N}\Lie(C^r(K,\G))
        \cong
        C^\infty(K,\Lie(\G)).
    \]
\end{example}

\begin{example}[Gauge groupoids and Atiyah algebroids]
\label{ex:gauge-groupoids}
    Let
    $(\pi_i\colon P_i\to M_i,G_i,r_{ij},f_{ij},h_{ij})$
    be a strict inverse system of principal bundles with strict
    principal atlas $\mathcal T$. Set
    $P:=\vpjl P_i$, $M:=\vpjl M_i$, and $G:=\vpjl G_i$.
    By Proposition~\ref{prop:strict-principal-bundles},
    $\pi:=\vpjl\pi_i\colon P\to M$ is a principal $G$-bundle.

    We first recall the construction of the gauge groupoid at each level; see~\cite{Mac05}. Put
    $Q_i:=(P_i\times P_i)/G_i$, where $G_i$ acts diagonally, write
    $\varpi_i\colon P_i\times P_i\to Q_i$ for the orbit map, and set
    $\src_i[p,q]:=\pi_i(q)$ and $\tgt_i[p,q]:=\pi_i(p)$.
    For arbitrary principal trivializations
    $\tau\colon\pi_i^{-1}(U)\to U\times G_i$ and
    $\lambda\colon\pi_i^{-1}(V)\to V\times G_i$, write
    $\wt\tau:=\pr_2\circ\tau$ and
    $\wt\lambda:=\pr_2\circ\lambda$, and define
    \[
        \Theta_{\tau,\lambda}[p,q]
        :=
        \bigl(\pi_i(p),\pi_i(q),
        \wt\tau(p)\wt\lambda(q)^{-1}\bigr).
    \]
    This is a bijection from
    $(\tgt_i,\src_i)^{-1}(U\times V)$ onto
    $U\times V\times G_i$, with inverse
    $(x,y,g)\mapsto
    [\tau^{-1}(x,g),\lambda^{-1}(y,e_i)]$.
    If $\tau',\lambda'$ are another pair of principal trivializations,
    then the overlap is
    $(x,y,g)\mapsto
        \bigl(x,y,
 c^{\tau'\tau}(x)g
        c^{\lambda'\lambda}(y)^{-1}\bigr).$
    Hence these ``gauge" charts define a smooth manifold structure on $Q_i$.

    In product principal coordinates the quotient map $\varpi_i$ is
    represented by
    $(x,a,y,b)\mapsto(x,y,ab^{-1})$. The coordinate change
    $(x,a,y,b)\mapsto(x,y,ab^{-1},b)$, with inverse
    $(x,y,g,b)\mapsto(x,gb,y,b)$, therefore represents $\varpi_i$ by
    projection onto the first three factors. Thus $\varpi_i$ is a
    surjective submersion.
    If $\src_i[p,q]=\tgt_i[p',q']$, there is a unique $g\in G_i$
    such that $p'=q\cdot g$; define
    $\mul_i([p,q],[p',q']):=[p\cdot g,q']$.
    This is independent of the representatives: replacing
    $(p,q)$ and $(p',q')$ by
    $(p\cdot a,q\cdot a)$ and $(p'\cdot b,q'\cdot b)$ replaces
    $g$ by $a^{-1}gb$ and the resulting pair by the diagonal action
    of $b$. Define also
    $\one_i(x):=[p,p]$, for $p\in\pi_i^{-1}(x)$, and
    $\inv_i[p,q]:=[q,p]$.

    Using gauge charts
    $\Theta_{\tau,\lambda}$,
    $\Theta_{\lambda,\mu}$, and
    $\Theta_{\tau,\mu}$ for the first arrow, second arrow, and product,
    respectively, one has
    $\mul_i((x,y,g),(y,z,h))=(x,z,gh)$.
    Likewise, using
    $\Theta_{\tau,\lambda}$ and $\Theta_{\lambda,\tau}$,
    $\inv_i(x,y,g)=(y,x,g^{-1})$, while
    $\src_i(x,y,g)=y$, $\tgt_i(x,y,g)=x$, and in
    $\Theta_{\tau,\tau}$ one has $\one_i(x)=(x,x,e_i)$.
    Hence $Q_i\rightrightarrows M_i$ is a Lie groupoid.

    For $i\leq j$, define
    $q_{ij}\colon Q_j\to Q_i$ by
    $q_{ij}[p,q]:=[r_{ij}(p),r_{ij}(q)]$.
    Equivariance of $r_{ij}$ makes this well defined, and the
    inverse-system identities for the $r_{ij}$ imply those for the
    $q_{ij}$. Moreover,
    $q_{ij}\circ\varpi_j=\varpi_i\circ(r_{ij}\times r_{ij}).$
    This also shows that $q_{ij}$ is smooth. Indeed, on a gauge-chart
    domain the map
    $(x,y,g)\mapsto
    (\tau^{-1}(x,g),\lambda^{-1}(y,e_j))$
    is a smooth local section of $\varpi_j$, and composing this
    section with the preceding identity gives a local smooth
    expression for $q_{ij}$. The maps $q_{ij}$ are groupoid
    morphisms: if $p'=q\cdot g$, then
    $r_{ij}(p')=r_{ij}(q)\cdot h_{ij}(g)$, which gives compatibility
    with multiplication; the other structural maps are immediate.

    We now use the strict principal atlas. Fix
    $(z_i)\in\vpjl Q_i$ and put
    $x_i:=\tgt_i(z_i)$ and $y_i:=\src_i(z_i)$. Choose
    $\{\tau_i\},\{\lambda_i\}\in\mathcal T$ about $(x_i)$ and
    $(y_i)$, respectively. In the corresponding gauge charts,
    $\Theta_{\tau_i,\lambda_i}\circ q_{ij}\circ
        \Theta_{\tau_j,\lambda_j}^{-1}=
 f_{ij}\times f_{ij}\times h_{ij}.$
    Thus, if
    $\Theta_{\tau_i,\lambda_i}(z_i)=(x_i,y_i,g_i)$, then
    $h_{ij}(g_j)=g_i$. Combining the gauge charts with \sis charts
    on the two base factors and on the group factor gives compatible
    charts on $\{Q_i\}$.

    For two such gauge-chart families the inverse-limit transition map
    has the form
    $ (x,y,g)\mapsto
        \bigl(x,y,
        c^{\tau'\tau
        }(x)g
        c^{\lambda'\lambda}(y)^{-1}\bigr).$  
     By the defining property of $\mathcal T$, the compatible limit transition functions are continuous; hence they are morphisms of strict inverse systems and are smooth locally by Lemma~\ref{lemma:basic-C^k}. Smoothness of the group operations then shows that the gauge transition map is smooth.
     Its continuity also shows
    that the inverse-limit overlap domains are open. Hence these
    gauge chart families form a strict atlas on $\{Q_i\}$.

    In these charts the source maps are projections, so
    $\{\src_i\}$ is split-compatible. Using matched \sis gauge
    charts as above, multiplication is represented by
    $((x,y,g),(y,z,h))\mapsto(x,z,gh)$, and inversion by
    $(x,y,g)\mapsto(y,x,g^{-1})$. Thus multiplication and inversion
    are morphisms of strict inverse systems. Hence
    $\{Q_i\rightrightarrows M_i\}$ is a strict inverse system of Lie
    groupoids.

    Applying the preceding levelwise construction to the principal
    bundle $P\to M$, form the gauge groupoid
    $(P\times P)/G\rightrightarrows M$. There is a canonical map
    $\Gamma\colon(P\times P)/G\to\vpjl Q_i$ given by
    $\Gamma([(p_i)_i,(q_i)_i])=([p_i,q_i])_i$.
    Since inverse limits do not in general commute with quotients, we
    verify directly that $\Gamma$ is bijective.

    Suppose first that two compatible pairs $(p_i,q_i)$ and
    $(p_i',q_i')$ have the same image in every $Q_i$. There is a
    unique $g_i\in G_i$ such that
    $(p_i',q_i')=(p_i\cdot g_i,q_i\cdot g_i)$. Compatibility and
    freeness give $g_i=h_{ij}(g_j)$, so $(g_i)\in G$ and the two
    pairs determine the same element of $(P\times P)/G$.

    Conversely, let $(z_i)\in\vpjl Q_i$ and put
    $x_i:=\tgt_i(z_i)$ and $y_i:=\src_i(z_i)$. Choose
    $\{\tau_i\},\{\lambda_i\}\in\mathcal T$ about $(x_i)$ and
    $(y_i)$ and write
    $\Theta_{\tau_i,\lambda_i}(z_i)=(x_i,y_i,g_i)$.
    Then $h_{ij}(g_j)=g_i$. Hence
    $p_i:=\tau_i^{-1}(x_i,g_i)$ and
    $q_i:=\lambda_i^{-1}(y_i,e_i)$ form compatible families and
    $[p_i,q_i]=z_i$. Thus $\Gamma$ is surjective. In the corresponding
    limit gauge charts $\Gamma$ is represented by the identity.
    Therefore
    \[
        (\vpjl P_i\times \vpjl P_i)/(\vpjl G_i)\cong\vpjl Q_i
    \]
    canonically as Lie groupoids.

    We finally identify the corresponding Lie algebroids directly.
    Put $A_i:=TP_i/G_i$, where $G_i$ acts on $TP_i$ by the
    differential of the principal action, and put
    $\mathfrak g_i:=\Lie(G_i)$. If
    $\omega_i^R\colon TG_i\to\mathfrak g_i$ denotes the right
    Maurer-Cartan map, a principal trivialization $\tau$ identifies
    $A_i|_U$  with
    $TU\oplus(U\times\mathfrak g_i)$ (as sets) by 
       $ [v_p]\mapsto
        \bigl(T\pi_i(v_p),
        \omega_i^R(d\wt\tau(v_p))\bigr).$
        Define
    $\alpha_i\colon A_i\to\Lie(Q_i)$ by
    $\alpha_i([v_p]):=
    d_{(p,p)}\varpi_i(v_p,0)$.
    This is well defined since
    $\varpi_i\circ(R_g\times R_g)=\varpi_i$. In the gauge chart
    $\Theta_{\tau,\tau}$ it identifies $(v,\xi)$ with the
    source-vertical tangent vector $(v,0,\xi)$ at the unit. Hence
    $\alpha_i$ is a smooth vector bundle isomorphism.

    Transport the Lie algebroid structure of $\Lie(Q_i)$ through
    $\alpha_i$ to $A_i$. We identify this structure explicitly. Its
    anchor is
    $\anch_i([v_p])=d\pi_i(v_p)$. Moreover, a section of $A_i$
    corresponds uniquely to a $G_i$-invariant vector field $X$ on
    $P_i$, and under $\alpha_i$ the corresponding right-invariant
    source-vertical vector field on $Q_i$ is
    $\widehat X([p,q]):= d_{(p,q)}\varpi_i(X(p),0).$
        Since $(X,0)$ and $(Y,0)$ are $\varpi_i$-related to
    $\widehat X$ and $\widehat Y$, respectively,
    $[\widehat X,\widehat Y]=\widehat{[X,Y]}$. Thus the transported
    bracket on $A_i$ is precisely the bracket induced by
    $G_i$-invariant vector fields on $P_i$. Consequently
    $\alpha_i\colon TP_i/G_i\to\Lie(Q_i)$
    is a canonical Lie algebroid isomorphism.

    Define
    $\overline r_{ij}\colon TP_j/G_j\to TP_i/G_i$ by
    $\overline r_{ij}[v_p]:=[dr_{ij}(v_p)]$. This is well defined by
    equivariance of $r_{ij}$. Differentiating
    $q_{ij}\circ\varpi_j
    =\varpi_i\circ(r_{ij}\times r_{ij})$ gives
    $\alpha_i\circ\overline r_{ij} = \Lie(q_{ij})\circ\alpha_j.$
    Hence the maps $\overline r_{ij}$ are smooth Lie algebroid
    morphisms. Using the natural family $\{\alpha_i\}$, the strict
    inverse-system structure on $\{\Lie(Q_i)\}$ given by
    Theorem~\ref{thm:inv-lie-groupoid} induces the corresponding
    strict inverse-system structure on $\{TP_i/G_i\}$.

    The same construction for the limit principal bundle gives
    $\Lie((P\times P)/G)\cong TP/G$. Therefore
    Theorem~\ref{thm:inv-lie-groupoid} and the preceding
    identifications give canonical Lie algebroid isomorphisms
    \[
        T(\vpjl P_i)/\vpjl G_i
        \cong
        \Lie((\vpjl P_i\times \vpjl P_i)/\vpjl G_i)
        \cong
        \vpjl\Lie(Q_i)
        \cong
        \vpjl(TP_i/G_i).
    \]

    Finally, if $\{G_i\}$ is strong, the transition condition in the
    definition of a strict principal atlas is automatic. If, in
    addition, $\{M_i\}$ is strong, then
    Proposition~\ref{prop:strict-principal-bundles} gives that
    $\{P_i\}$ is strong, and the gauge charts above show that
    $\{Q_i\}$ is strong as well.

    For inverse sequences, suppose in addition that the group and
    base systems are strong, every $P_n$ is simply connected, and
    each $r_{n,n+1}\colon P_{n+1}\to P_n$ is a locally trivial bundle.
    Fix $x=(x_n)\in\vpjl M_n$ and choose $q=(q_n)\in P_x$.
    The pointed diffeomorphisms $P_n\to Q_{n,x_n}$,
    $p\mapsto[p,q_n]$, intertwine the connecting maps.
    Since the gauge groupoid system is strong, the source-fiber
    topology agrees with the inverse limit topology.
    Proposition~\ref{prop:source-fiber-homotopy} therefore gives
    $\pi_1((\vpjl Q_n)_x,1_x)\cong
    \vpjl\nolimits^1\pi_2(P_n,q_n)$.
    In particular, the limit gauge groupoid is source-simply
    connected if the sequence $\{\pi_2(P_n,q_n)\}$ satisfies the
    Mittag-Leffler condition for every $x$. This holds when
    the fibers of the maps $r_{n,n+1}$ are simply connected.
\end{example}

\begin{example}[Hopf gauge groupoids and Atiyah algebroids]
\label{ex:hopf-gauge}
    Let $K:=SU(2)$ and let
    $K\to S^7\xrightarrow{\pi}S^4$ be the quaternionic Hopf
    principal bundle. For $n\in\N_0$, put
    $M_n:=S^4\times\R^n$, $P_n:=S^7\times\R^n$, and
    $\pi_n:=\pi\times\id_{\R^n}$. Let
    $f_{nm}\colon M_m\to M_n$ and
    $r_{nm}\colon P_m\to P_n$ be the canonical projections.
    Put $\G_n:=(P_n\times P_n)/K\rightrightarrows M_n$ and
    $\F_{nm}[p,q]:=[r_{nm}(p),r_{nm}(q)]$.
    Writing $\G_0:=(S^7\times S^7)/K\rightrightarrows S^4$,
    we have canonical groupoid isomorphisms
    $\G_n\cong\G_0\times(\R^n\times\R^n
    \rightrightarrows\R^n)$, where the second factor is the
    pair groupoid. The connecting maps are the identity on
    $\G_0$ and the coordinate projections on the pair groupoids.
    If $\Theta_0\colon U_0\times W_0\to\Omega_0$ is a
    normalized source chart on $\G_0$, its product with the global
    source chart of the pair groupoid is
    $\Theta_n((y,a),(v,b))=(\Theta_0(y,v),(a+b,a))$.
    The corresponding charts at arbitrary arrows have domains
    $\Omega_n^{g_n}=\Omega_0^{g_0}\times\R^n\times\R^n
    =\F_{0n}^{-1}(\Omega_0^{g_0})$.
    Thus Corollary~\ref{cor:strong-integration} applies: the
    groupoids form a strong inverse sequence integrating the
    inverse sequence of Atiyah algebroids.

    Put $P:=S^7\times\R^\N$ and $M:=S^4\times\R^\N$. By
    Example~\ref{ex:gauge-groupoids} and
    Theorem~\ref{thm:inv-lie-groupoid}, there are canonical
    isomorphisms
    \[
        \vpjl\G_n\cong(P\times P)/K
        \qquad \text{and} \qquad
        \Lie(\vpjl\G_n)\cong TP/K\cong\vpjl(TP_n/K).
    \]
    Write $A:=TP/K\to M$ and $A_n:=TP_n/K\to M_n$.

    Fix $x=(x_n)\in M$ and choose a compatible
    $q=(q_n)\in P_x$, so $r_{nm}(q_m)=q_n$.
    The diffeomorphisms $P_n\to\G_{n,x_n}$,
    $p\mapsto[p,q_n]$, intertwine the connecting maps. Thus the
    source-fiber inverse sequence is $S^7\times\R^n$ with the standard
    projections, whose fibers are Euclidean spaces, and its inverse
    limit is $S^7\times\R^\N$. Since the groupoid system is strong,
    the canonical bijection from the limit source fiber to the inverse
    limit of the level source fibers is a homeomorphism. Hence the
    hypotheses of Proposition~\ref{prop:source-fiber-homotopy} hold.
    Moreover, $\pi_1(S^7\times\R^n)=0$ and
    $\pi_2(S^7\times\R^n)=0$, so
    $\vpjl{}^1\pi_2(S^7\times\R^n)=0$.
    Corollary~\ref{cor:ssc-limit} therefore shows that the
    inverse-limit gauge groupoid is source-simply connected.

    Let $\xi\to S^4$ be the complex rank-two bundle associated to the
    defining representation of $SU(2)$. Its second Chern number is
    nonzero~\cite[Lemma~20.9]{MS74}, so the Hopf bundle admits no
    flat principal connection. A vector bundle morphism
    $s_0\colon TS^4\to A_0$ satisfying
    $\anch_0\circ s_0=\id_{TS^4}$ is a principal connection, and it is
    a Lie algebroid morphism precisely when the connection is flat;
    see~\cite[\S5.2,\S5.3]{Mac05}. Thus no such $s_0$ can be a Lie
    algebroid morphism.

    Let $\rho\colon A\to A_0$ be the Lie algebroid morphism induced by
    the principal-bundle projection
    $S^7\times\R^\N\to S^7$, and let
    $\iota\colon S^4\to M$ be given by $\iota(x)=(x,0)$.
    If there were a Lie algebroid morphism
    $s\colon TM\to A$ satisfying $\anch\circ s=\id_{TM}$, then
    $\rho\circ s\circ T\iota\colon TS^4\to A_0$ would be a Lie
    algebroid morphism satisfying
    $\anch_0\circ\rho\circ s\circ T\iota=\id_{TS^4}$, contradicting
    the preceding paragraph. Thus the Atiyah anchor of $A$ is
    surjective but admits no Lie algebroid splitting.
\end{example}

\end{document}